\documentclass[12pt]{article}

\usepackage{amssymb, amsthm, amsmath}

\usepackage{parskip}
\makeatletter 
\def\thm@space@setup{%
 \thm@preskip=\parskip \thm@postskip=0pt
}
\def\th@remark{%
  \thm@headfont{\itshape}%
  \normalfont 
  \thm@preskip\parskip \thm@postskip=0pt
}
\makeatother

\usepackage[]{hyperref}
\usepackage[initials,nobysame,msc-links]{amsrefs}
\DefineSimpleKey{bib}{how}
\renewcommand{\PrintDOI}[1]{%
  \href{http://dx.doi.org/#1}{{\tt DOI:#1}}%
}
\makeatletter
\renewcommand{\PrintDatePV}[1]{%
    \let\@tempa\PrintDate
    \@tempa{#1}%
}
\makeatother
\renewcommand{\eprint}[1]{#1}
\BibSpec{book}{%
    +{}  {\PrintPrimary}                {transition}
    +{.} { \PrintDate}                  {date}
    +{.} { \emph}                     {title}
    +{.} { }                            {part}
    +{:} { \emph}                     {subtitle}
    +{,} { \PrintEdition}               {edition}
    +{}  { \PrintEditorsB}              {editor}
    +{,} { \PrintTranslatorsC}          {translator}
    +{,} { \PrintContributions}         {contribution}
    +{,} { }                            {series}
    +{,} { \voltext}                    {volume}
    +{,} { }                            {publisher}
    +{,} { }                            {organization}
    +{,} { }                            {address}
    +{,} { }                            {status}
    +{,} { \PrintDOI}                   {doi}
    +{,} { \PrintISBNs}                 {isbn}
    +{}  { \parenthesize}               {language}
    +{}  { \PrintTranslation}           {translation}
    +{;} { \PrintReprint}               {reprint}
    +{.} { }                            {note}
    +{.} {}                             {transition}
    +{}  {\SentenceSpace \PrintReviews} {review}
}
\BibSpec{article}{%
    +{}  {\PrintAuthors}                {author}
    +{,} { \emph}                     {title}
    +{.} { }                            {part}
    +{:} { \emph}                     {subtitle}
    +{,} { \PrintContributions}         {contribution}
    +{.} { \PrintPartials}              {partial}
    +{,} { }                            {journal}
    +{}  { \textbf}                     {volume}
    +{}  { \PrintDatePV}                {date}
    +{,} { \issuetext}                  {number}
    +{,} { \eprintpages}                {pages}
    +{,} { }                            {status}
    +{,} { \PrintDOI}                   {doi}
    +{,} { \eprint}        {eprint}
    +{}  { \parenthesize}               {language}
    +{}  { \PrintTranslation}           {translation}
    +{;} { \PrintReprint}               {reprint}
    +{.} { }                            {note}
    +{.} {}                             {transition}
    +{}  {\SentenceSpace \PrintReviews} {review}
}
\BibSpec{collection.article}{%
    +{}  {\PrintAuthors}                {author}
    +{,} { \emph}                     {title}
    +{.} { }                            {part}
    +{:} { \emph}                     {subtitle}
    +{,} { \PrintContributions}         {contribution}
    +{,} { \PrintConference}            {conference}
    +{}  {\PrintBook}                   {book}
    +{,} { }                            {booktitle}
    +{,} { \PrintDateB}                 {date}
    +{,} { pp.~}                        {pages}
    +{,} { }                            {publisher}
    +{,} { }                            {organization}
    +{,} { }                            {address}
    +{,} { }                            {status}
    +{,} { \PrintDOI}                   {doi}
    +{,} { \eprint}        {eprint}
    +{}  { \parenthesize}               {language}
    +{}  { \PrintTranslation}           {translation}
    +{;} { \PrintReprint}               {reprint}
    +{.} { }                            {note}
    +{.} {}                             {transition}
    +{}  {\SentenceSpace \PrintReviews} {review}
}
\BibSpec{misc}{%
  +{}{\PrintAuthors}  {author}
  +{,}{ \emph}      {title}
  +{.}{ }             {how}
  +{}{ \parenthesize} {date}
  +{,} { available at \eprint}        {eprint}
  +{,}{ available at \url}{url}
  +{,}{ }             {note}
  +{.}{}              {transition}
}

\newcommand{\Z}{\mathbb{Z}}
\newcommand{\N}{\mathbb{N}}
\newcommand{\R}{\mathbb{R}}
\newcommand{\C}{\mathbb{C}}
\newcommand{\F}{\mathbb{F}}
\newcommand{\D}{\mathcal{D}}
\newcommand{\T}{\mathbb{T}}
\newcommand{\FO}{\mathbb{F}\mathrm{O}}
\newcommand{\FU}{\mathbb{F}\mathrm{U}}
\newcommand{\U}{\mathrm{U}}
\newcommand{\SU}{\mathrm{SU}}
\newcommand{\SO}{\mathrm{SO}}
\newcommand{\SL}{\mathrm{SL}}
\newcommand{\Up}{\mathrm{U}^+}
\newcommand{\Op}{\mathrm{O}^+}
\newcommand{\Hsp}{\mathcal{H}}
\newcommand{\Ksp}{\mathcal{K}}
\newcommand{\id}{\mathrm{id}}

\newcommand{\cb}{\mathrm{cb}}
\newcommand{\opos}{\mathrm{op}}
\newcommand{\G}{\mathbb{G}}
\newcommand{\qH}{\mathbb{H}}

\newcommand{\qbin}[3]{\left[ \begin{array}{c} #1 \\ #2 \end{array}\right]_{#3}}
\newcommand{\reg}[1]{\mathcal{#1}}
\newcommand{\Ad}{\mathrm{Ad}}

\newcommand{\ccdot}{\,\cdot\,}
\newcommand{\Dd}{\mathcal{D}}
\newcommand{\dDd}{\widehat{\mathcal{D}}}
\newcommand{\Pol}{\mathcal{O}}

\newcommand{\vN}{L}
\newcommand{\qAut}{\mathrm{A}^{\mathrm{aut}}}
\newcommand{\frprd}{\mathbin{*}}
\newcommand{\rfrprd}{\mathbin{*}_r}
\usepackage{wasysym}
\newcommand{\absv}[1]{\left\lvert #1 \right\rvert}
\newcommand{\norm}[1]{\left\lVert #1 \right\rVert}

\DeclareMathOperator{\Tr}{Tr}
\DeclareMathOperator{\Ind}{Ind}

\DeclareMathOperator{\Rep}{Rep}

\DeclareMathOperator{\Irr}{Irr}
\DeclareMathOperator{\sgn}{sgn}
\DeclareMathOperator{\Aut}{Aut}

\DeclareMathOperator{\Sd}{Sd}

\newcommand{\eps}{\varepsilon}
\newcommand{\al}{\alpha}
\newcommand{\be}{\beta}
\newcommand{\ot}{\otimes}
\newcommand{\vphi}{\varphi}
\newcommand{\om}{\omega}
\newcommand{\Inn}{\operatorname{Inn}}

\theoremstyle{plain}
\newtheorem{Theorem}{Theorem}
\newtheorem{Lem}[Theorem]{Lemma}
\newtheorem{Prop}[Theorem]{Proposition}
\newtheorem{Cor}[Theorem]{Corollary}

\theoremstyle{definition}
\newtheorem{Def}[Theorem]{Definition}

\theoremstyle{remark}
\newtheorem{Rem}[Theorem]{Remark}

\title{CCAP for universal discrete quantum groups}

\date{}

\begin{document}

\author{Kenny De Commer\thanks{D\'epartement de math\'ematiques, Universit\'e de Cergy-Pontoise, UMR CNRS 8088, F-95000 Cergy-Pontoise, France, email: {\tt Kenny.De-Commer@u-cergy.fr}}
\and Amaury Freslon\thanks{Universit\'e Paris Diderot, Sorbonne Paris Cit\'e, UMR 7586, 8 place FM/13, 75013, Paris, France, email: {\tt freslon@math.jussieu.fr}}
\and Makoto Yamashita\thanks{Institut for Matematiske Fag, K{\o}benhavns Universitet, Universitetsparken 5, 2100-K{\o}benhavn-\O, Denmark (on leave from Ochanomizu University),
email:
{\tt yamashita.makoto@ocha.ac.jp}}}

\maketitle

\begin{abstract}
\noindent We show that the discrete duals of the free orthogonal quantum groups have the Haagerup property and the completely contractive approximation property. Analogous results
hold for the free unitary quantum groups and the quantum automorphism groups of finite-dimensional C$^*$-algebras. The proof relies on the monoidal equivalence between free
orthogonal quantum groups and $\SU_q(2)$ quantum groups, on the construction of a sufficient supply of bounded central functionals for $\SU_q(2)$ quantum groups, and on the free
product techniques of Ricard and Xu. Our results generalize previous work in the Kac setting due to Brannan on the Haagerup property, and due to the second author on the CCAP.
\end{abstract}


AMS 2010 \emph{Mathematics subject classification}: 46L09, 46L65.

\section*{Introduction}

The universal compact quantum groups $\Op_F$ and $\Up_F$, called respectively the \emph{free orthogonal} and \emph{free unitary} quantum groups, were introduced by Van Daele and Wang~\cite{VDae2} as nonclassical compact quantum groups characterized by certain universality property with prescribed actions. Banica~\cite{Ban2} carried out a detailed analysis of the structure of the representation category and C$^*$-algebras associated to these quantum groups. Regarding approximation properties, Vaes and Vergnioux~\cite{Vae1} obtained, from a study of the action of the discrete dual $\FO_F$ on its Martin boundary, the exactness and the Akemann--Ostrand property of the reduced C$^*$-algebra $C_r^*(\FO_F)$. This allowed them to infer the generalized solidity of the von Neumann algebra $\vN(\FO_F)$ generated by $C^*_r(\FO_F)$.

In this paper, we focus on the operator space structure of universal compact quantum groups.  Recall that a C$^*$-algebra $A$ is said to have the \emph{completely contractive approximation property} (CCAP for short, also known as the \emph{complete metric approximation property}, or
being
\emph{weakly amenable with Cowling--Haagerup constant $1$}) if there is a net of completely contractive finite rank operators on $A$ approximating the identity in the topology of
pointwise norm-convergence~\cite{Cow1}. Analogously, if $A$ is a von Neumann algebra, the approximating maps are normal, and the convergence is with respect to the pointwise
$\sigma$-weak convergence, then $A$ is said to have the W$^*$CCAP.

Our attention will mostly be concentrated on the free orthogonal quantum groups $\Op_F$ with $F$ a complex matrix of size $N$ satisfying $F \overline{F} \in \R I_{N}$.  Our main result, Theorem \ref{Thm:AoF-central-approx}, implies that the reduced C$^*$-algebra $C^*_r(\FO_F)$, which we will also denote as $C_r(\Op_F)$, has the CCAP\@. This generalizes previous results in the Kac-type case where $F$ is the identity matrix $I_{N}$. Namely, Brannan~\cite{Bra1} showed that the discrete quantum group $\FO_{N}$ has the Haagerup property and the metric approximation property. Using a precursory idea of this paper, the second author showed the CCAP of $\FO_F$ when $\Op_F$ is monoidally equivalent to a Kac one $\Op_m$~\cite{Fre1}.   The approximation maps for the CCAP at hand are given by \emph{central} multipliers close to completely positive maps. Such operators behave well under monoidal equivalence, passing to discrete quantum subgroups and taking free products (the latter by adapting the technique of Ricard and Xu~\cite{Ric1}). Combining this with results due to Banica~\cite{Ban6}, Wang, and Brannan, this allows us to prove the CCAP for \emph{all} $\FO_F$, all $\FU_F$, as well as the duals of all free automorphism groups of finite-dimensional C$^*$-algebras with fixed states. Also the W$^*$CCAP for the associated von Neumann algebras follows.

The CCAP for $\FO_F$ is interesting from several viewpoints. From the viewpoint of quantum group theory, one is naturally led to the comparison of $\FO_F$ with the free groups.  Ever since the seminal work of Haagerup~\cite{Haa1}, it has been found out that the free groups $\F_n$ with $n\geq 2$ enjoy many elaborate approximation properties although they fail to be amenable, and it is natural to expect similar properties for $\FO_F$.  However, the lack of unimodularity brings in several obstacles if one tries to apply the established methods developed for $\F_n$. From the viewpoint of deformation rigidity theory,
W$^*$CCAP and its variants are exploited to show the sparsity of Cartan subalgebras in the breakthrough papers by Ozawa and
Popa~\cite{Oza1}, and by Popa and Vaes~\cite{Pop1}.  Regarding $\FO_F$, recent work of Isono~\cite{Iso1} on a strengthened Akemann--Ostrand property established that the von Neumann
algebra $\vN(\FO_F)$ has no Cartan subalgebra provided it has the W$^*$CCAP and is not injective.

Our proof of the CCAP uses a holomorphicity argument on the Banach space of completely bounded multipliers. Roughly speaking,
we proceed by finding a family of complex numbers $(b_{d}(z))_{d\in\N}$ for $z \in \D = \{z\in \C\mid -1<\Re(z)<1\}$ satisfying the following conditions.
\begin{enumerate}
\item \label{it:cp-mult-condi}
Denote by $p_{d}$ the projection of $C_r(\Op_F)$ onto the isotypic subspace of the irreducible representation with spin $d/2$. We require that for each $z\in \D$, there exists a completely bounded map $\Psi_z$ on $C_r(\Op_F)$ such that $\Psi_z = \sum_{d} b_{d}(z) p_{d}$ as a pointwise norm-limit.
\item \label{it:approx-id-condi}
The map $z \mapsto \Psi_{z}$ is holomorphic in $z$ as a map into the Banach space of completely bounded maps.
\item
The operator $\Psi_t$ is ucp (unital completely positive) for $t \in (-1, 1)$, and converges to the identity as $t \rightarrow 1$ in the topology of pointwise norm-convergence.
\item \label{it:summability-condi}
There exists $0<C<1$ such that $\sum_d |b_d(t)|\norm{p_d}_\cb<\infty$ for all $0<t<C$.
\end{enumerate}
Once one has $(b_{d}(z))_{d, z}$ as above,
normalized truncations of the operators $\Psi_{t}$ can approximate $\Psi_t$ when $t$ is in the range $0<t<C$. Then, the analyticity implies that \emph{any}
$\Psi_t$ with $-1<t<1$ can be approximated by completely contractive finite rank operators.  In the context of discrete groups, the holomorphicity of completely bounded multipliers plays a central role in the work of Pytlik and Szwarc~\cite{Pyt1}, but it was Ozawa who conceived what we need, namely, that the holomorphicity could be an essential substitute for the Bozejko--Picardello type inequality, in his proof of weak amenability of hyperbolic groups (see~\cite{BrO1}*{second proof of Corollary~12.3.5}).

According to \cite{BDV1}, $\Op_F$ is monoidally equivalent to Woronowicz's compact quantum group $\SU_{q}(2)$ \cite{Wor3} for an appropriate choice of $q$. A simple argument based on the linking algebra of such an
equivalence shows that the multipliers of the form $\Psi_{z}$ can be transferred between the reduced C$^*$-algebras of monoidally equivalent quantum groups. This enables us to restrict the problem of finding the $b_d(z)$ to the case of $\SU_{q}(2)$.

In the Kac-type case, questions of the above sort may be tackled from a classical viewpoint. For example, central multipliers can be obtained by averaging arbitrary
multipliers~\cite{Kra1}. However, in the non-Kac-type case, the averaging affects the growth condition in a serious way if we try to retain the positivity~\cite{Cip1}. To find a proper $b_d(t)$ in the case of general $q$, we will use a one parameter family of representations of the Drinfel'd double of $\SU_{q}(2)$ given by Voigt in his study of the
Baum--Connes conjecture for $\FO_F$~\cite{Voi1}. We should also stress that our construction utilizes a purely quantum phenomenon, which degenerates in the classical case of
$\SU(2)$.

The final part of this paper is devoted to a short study of the relation between the central states and the Drinfel'd double construction, aiming for a more conceptual understanding
of the above mentioned quantum phenomenon.  It turns out that the complete positivity of our multipliers is related to the nonclassical spectrum of characters of $\SU_q(2)$ in a
slightly disguised subalgebra of the Drinfel'd double.  In fact, even though $\SU_q(2)$ is coamenable, the double is no longer amenable, and this allows us to ultimately overcome the
problem of the characters of the free quantum groups having `too small' norm in the reduced algebra, as observed by Banica~\cite{Ban6}.

In the appendix by S.~Vaes, it is shown that all $L(\FU_F)$ are full factors.  As remarked above, this allows one to infer that the $L(\FU_F)$ are factors without Cartan subalgebra, for arbitrary invertible matrices $F$.

\subsection*{Acknowledgements}

M.Y. is supported by Danish National Research Foundation through the Centre for Symmetry and Deformation (DNRF92), and by JSPS KAKENHI Grant Number 25800058.  Part of this work was
done while K.D.C. and M.Y. were participating in the Focus program on Non-Commutative Geometry and Quantum groups, Fields institute, 3--28 June 2013. We would like to thank the
organizers and the staff for their hospitality. We also thank Y. Arano, Y. Isono, N. Ozawa, A. Skalski, and R. Vergnioux for their useful comments. We especially thank S. Vaes for including an unpublished note of his as an appendix to our paper.

\section{Preliminaries}\label{sec:prelim}

\subsection{Compact quantum groups}

Let $\G$ be a \emph{compact quantum group} \cites{Wor1,VDae1}. That is, we are given a unital Hopf $*$-algebra $(\Pol(\G),\Delta)$, together with a functional $\varphi$ satisfying
the invariance property
\[
(\id\otimes \varphi)\Delta(x) = \varphi(x)1 = (\varphi\otimes \id)\Delta(x),\quad x\in \Pol(\G)
\]
and the \emph{state} property: $\varphi(1) = 1$ and $\varphi(x^{*}x)\geqslant 0$ for all $x\in \Pol(\G)$. We denote the antipode of $\Pol(\G)$ by $S$. As usual, we use the Sweedler
notation $\Delta(x) = x_{(1)}\otimes x_{(2)}$ to express the coproduct in a convenient way.

The associated reduced C$^*$-algebra and von Neumann algebra are denoted respectively by $C_r(\G)$ and $\vN^\infty(\G)$. We also regard these operator algebras as the convolution
algebras of the dual \emph{discrete quantum group} $\widehat{\G}$, and hence also use the equivalent notations $C^*_r(\widehat{\G})$ and $\vN(\widehat{\G})$ interchangeably.

When $\G$ is a compact quantum group as above, a discrete quantum subgroup $\widehat{\qH}$ of $\widehat{\G}$ is given by a Hopf $*$-subalgebra $\Pol(\qH) \subset \Pol(\G)$.

The \emph{Woronowicz characters}~\cite{Wor1} are a distinguished family of unital algebra homomorphisms $(f_{z})_{z \in \C}$ from $\Pol(\G)$ into $\C$ satisfying the rules
\[
f_{z}(x_{(1)})f_{w}(x_{(2)}) = f_{z+w}(x),\quad f_{z}(x^{*}) = \overline{f_{-\bar{z}}(x)},\qquad (x,y\in \Pol(\G), w,z\in \C).
\]
The antipode squared $S^2$ then satisfies
\[
S^2(x) = f_{1}(x_{(1)})x_{(2)}f_{-1}(x_{(3)}).
\]

The structure of $\Pol(\G)$ can be captured by components of irreducible unitary corepresentations. That is, there is a basis $(u^{(\pi)}_{i j})_{\pi, i, j}$ of $\Pol(\G)$ satisfying
\begin{align*}
\sum_{k} u^{(\pi)}_{i k} u^{(\pi) *}_{j k} &= \delta_{i j} = \sum_{k} u^{(\pi) *}_{k i} u^{(\pi)}_{k j},&
\Delta\bigl(u^{(\pi)}_{i j}\bigr) &= \sum_k u^{(\pi)}_{i k} \otimes u^{(\pi)}_{k j},&
S\bigl(u^{(\pi)}_{i j}\bigr) &= u^{(\pi) *}_{j i}.
\end{align*}
One may moreover choose the $u_{ij}^{(\pi)}$ so that $f_z$ satisfies $f_{z}(u^{(\pi)}_{i j}) = \delta_{i j} Q_{\pi, i}^{z}$ for some positive real numbers $Q_{\pi, i}$. In that
case, one has
\begin{align}\label{eq:inn-prod}
\varphi\bigl(u^{(\pi')}_{i j} u^{(\pi) *}_{k l}\bigr) &= \frac{\delta_{\pi,\pi'}\delta_{j l} \delta_{k i} Q_{\pi, j}}{\dim_{q}(\pi)}, &
\varphi\bigl(u^{(\pi) *}_{i j} u^{(\pi')}_{k l}\bigr) &= \frac{\delta_{\pi,\pi'}\delta_{i
k} \delta_{j l} Q_{\pi, i}^{-1}}{\dim_{q}(\pi)},
\end{align}
where $\dim_q(\pi)= \sum_i Q_{\pi, i} = \sum_i Q_{\pi, i}^{-1}$.

When $\pi$ is an irreducible representation of $\G$, the associated \emph{isotypic projection} $p_{\pi}$ is defined by $p_{\pi}(u_{ij}^{(\pi')}) = \delta_{\pi,\pi'}u_{ij}^{(\pi)}$.
It extends to a completely bounded map on $C_r(\G)$, since one can write
\begin{equation}\label{eq:proj-pi-as-fourier-transform}
p_{\pi}(x) = (\id\otimes \varphi)(\Delta(x)(1\otimes c_{\pi}^*)),
\end{equation}
where  $c_{\pi}= \dim_q(\pi)\sum_{i}Q_{\pi, i}^{-1}u_{ii}^{(\pi)}$.

\subsection{Monoidal equivalence}

Two compact quantum groups $\G_{1}$ and $\G_{2}$ are said to be \emph{monoidally equivalent} if their representation categories are equivalent as abstract tensor C$^*$-categories.
This implies the existence of $*$-algebras $\Pol(\G_{r}, \G_{s})$ for $r, s \in \{1, 2\}$ and injective $*$-homomorphisms
\begin{equation}\label{eq:monoidal-equiv-linking-homs}
\Delta_{rs}^{t}\colon \Pol(\G_{r}, \G_{s}) \rightarrow \Pol(\G_{r}, \G_{t}) \otimes \Pol(\G_{t}, \G_{s})
\end{equation}
satisfying obvious coassociativity conditions, with $(\Pol(\G_{r}, \G_{r}), \Delta_{rr}^{r}) = (\Pol(\G_{r}), \Delta_{r})$ and the coactions $\Delta_{r s}^{r}$ and
$\Delta_{rs}^{s}$ being ergodic~\cite{BDV1}. Each of these $*$-algebras then contains an isotypic subspace $\Pol(\G_{r}, \G_{s})_{\pi}$ associated with each irreducible representation $\pi$ of
$\Rep(\G_{1}) \cong \Rep(\G_{2})$.

The ergodicity of the coactions implies that the $*$-algebras $\Pol(\G_{r}, \G_{s})$ have distinguished states given by averaging with the Haar state of $\G_{r}$ (or, equivalently,
of
$\G_{s}$). The GNS-construction produces reduced C$^*$-algebras $C_r(\G_{r}, \G_{s})$ and \emph{injective}
$*$-homomorphisms extending~\eqref{eq:monoidal-equiv-linking-homs}, where the codomain is understood to be the minimal tensor product of the relevant C$^*$-algebras.

\subsection{Quantum \texorpdfstring{$\SU(2)$}{SU(2)} groups}\label{SecQuantumSU}

Let $q$ be a nonzero real number with $|q|\leq 1$.
In this paper we use the following presentation of $\SU_q(2)$. The Hopf $^*$-algebra $\Pol(\SU_{q}(2))$ is generated by two elements $\alpha$ and $\gamma$ with the sole requirement that the matrix
\[
U_{1/2} = \begin{pmatrix} \alpha & -q\gamma^{*}\\ \gamma & \alpha^{*}\end{pmatrix}
\]
is a unitary corepresentation. The whole family of irreducible corepresentations of $\Pol(\SU_{q}(2))$ is then parametrized by half-integer weights $d / 2$ for $d \in \N$. Concretely, the one labeled by $d / 2$ is represented on
a Hilbert space $H_{d/2}$ of dimension $d + 1$, whose basis is labeled by the half-integers $-d/2, - d/2 + 1, \ldots, d/2$. We write the associated unitary matrix as $U_{d / 2} =
[u^{(d / 2)}_{i j}] \in B(H_{d / 2}) \otimes \Pol(\SU_q(2))$. With respect to this indexing, the operator $Q_{d/2}$ can be diagonalized as $Q_{d / 2, i} = \absv{q}^{2 i}$.

Let $(\mu_{d})_{d \in \N}$ denote \emph{the dilated Chebyshev polynomials of the second kind}, determined by the recurrence relation $\mu_{0}(x) = 1$, $\mu_{1}(x) = x$ and $x
\mu_k(x) = \mu_{k-1}(x)+\mu_{k+1}(x)$ for $k\geq1$. Then, the quantum dimension of $U_{d/2}$ is given by
\begin{equation}\label{eq:SUq2-q-dim-formula}
\dim_q(U_{d/2}) = \sum_{j = - d / 2}^{d / 2} \absv{q}^{2 j} = \mu_{d}(\absv{q+q^{-1}})
\end{equation}
for $d \in \N$, while the classical dimension satisfies $\dim(U_{d/2}) = \mu_d(2)$~\cite{Wor3}.

As $SU_q(2)$ is co-amenable, its associated reduced and universal C$^*$-algebra coincide, and will be denoted by $C(SU_q(2))$. It admits a faithful $*$-representation $\rho_q$ on $\ell^2(\N)\otimes \ell^2(\Z)$, defined by
\begin{align*}
\rho_q(\alpha)e_n\otimes e_k &= \sqrt{1-q^{2n}} e_{n-1}\otimes e_k,&
\rho_q(\gamma) e_n\otimes e_k &= q^n e_n\otimes e_{k-1}.
\end{align*}

The C$^*$-subalgebra of $C(\SU_q(2))$ generated by $\alpha$ is isomorphic to the Toeplitz algebra. It can be characterized as the universal unital C$^*$-algebra with a
single generator $\alpha$ satisfying the commutation relation $\alpha\alpha^*-q^2\alpha^*\alpha  = 1-q^2 $ (see \cite{Haj1}*{Section 3}). By considering only the first leg of
$\rho_q$, we get a faithful representation of $C^*(\alpha)$ on $\ell^2(\N)$, denoted by $\rho^0_q$.

\subsection{Free orthogonal quantum groups}

Let $N \geqslant 2$ be an integer, and let $F$ be an invertible complex matrix of size $N$. Let $\Pol(\Op_F)$ be the universal $*$-algebra generated by $N^{2}$ elements $(u_{i
j})_{1\leqslant i, j \leqslant N}$ such that the matrix $U = [u_{i j}]$ is unitary and $U = F \overline{U} F^{-1}$. Here, $\overline{U}$ denotes the component-wise adjoint matrix
$[{u_{i j}}^*]$ of $U$. Then $\Pol(\Op_F)$ has the structure of a Hopf $*$-algebra with the coproduct defined by
\[
\Delta(u_{i j}) = \sum_{k} u_{i k}\otimes u_{k j}.
\]
This determines a compact quantum group called the \emph{free orthogonal quantum group} $\Op_F$, introduced by Wang and Van Daele \cites{Wan1, VDae2}. The change of parameters $F \rightarrow
\lambda V F V^t$ for $\lambda \in \C^\times$ and $V \in \U(N)$ gives a naturally isomorphic quantum group. We will interchangeably denote the associated reduced C$^*$-algebra by
$C_r(\Op_F)$ and  $C_r^*(\FO_F)$, where $\FO_F$ stands for the discrete quantum group dual of $\Op_F$.

When $F$ satisfies $F \overline{F} = \pm I_N$, the compact quantum group $\Op_F$ is monoidally equivalent to $\SU_q(2)$ for the $q$ satisfying $q + q^{-1} =
\mp\Tr(F^{*}F)$~\cite{BDV1}. In particular, the irreducible representations of $\Op_F$ are also labeled by half-integers, and the associated quantum dimensions are the same as in the
$\SU_q(2)$-case.

\section{Central linear functionals}\label{SecNot}

Let $\G$ be a compact quantum group. The space $\Pol(\G)^{*}$ of linear functionals on $\Pol(\G)$ has the structure of a unital associative $*$-algebra defined by
\begin{align}\label{eq:dual-alg-structure}
(\phi \psi) (x) &= (\psi \otimes \phi)\circ\Delta (x),&
\omega^{*}(x) &= \overline{\omega(S(x^{*}))}.
\end{align}
Note that we take the opposite of the usual convolution product, to have nicer formulas in Section \ref{SecDrin}.

\begin{Def}
A functional on $\Pol(\G)$ is called \emph{central} if it commutes with every element in $\Pol(\G)^{*}$.
\end{Def}

Clearly $\omega$ is central if and only if one has
\[
\omega(x_{(1)})x_{(2)} = \omega(x_{(2)})x_{(1)},\quad \forall x\in \Pol(\G).
\]

It is easily seen that there is a one-to-one correspondence between central functionals and functions $\pi \rightarrow \omega_{\pi}\in \C$ on the set of equivalence classes of
irreducible representations of $\G$. Namely, to any central functional $\omega$ one associates the numbers
\begin{equation}\label{EqId}
\omega_{\pi} = \frac{1}{\dim(\pi)}\sum_i \omega \bigl(u_{ii}^{(\pi)}\bigr) \qquad \pi \in \Irr(\G),
\end{equation}
while to any $\pi \rightarrow \omega_{\pi}$ one associates the central functional $u_{ij}^{(\pi)}\mapsto \delta_{ij} \omega_{\pi}$. In the following, we will use this identification
without further comment.

For $\omega$ central, the slice map $T_{\omega} = (\omega \otimes \imath)\circ\Delta$ on $\Pol(\G)$ acts by multiplication by the scalar $\omega_{\pi}$ on the $\pi$-isotypic
component. Thus, $T_{\omega}(x)$ can be written as $\sum_{\pi} \omega_{\pi} p_{\pi}(x)$, which is a finite sum for any $x \in \Pol(\G)$. Conversely, any operator of the form $\sum_{\pi} \omega_{\pi} p_{\pi}$ is the slice map of the
central linear functional associated with $\pi \rightarrow \omega_{\pi}$. We call $T_{\omega}$ the \emph{central multiplier} associated with $\omega$.

We are interested in central multipliers which extend to completely bounded maps on $C_r(\G)$. These central multipliers $T_{\omega}$ enjoy particularly nice properties, analogous to
the Herz--Schur multiplier operators for locally compact groups. Such operators already appeared in several works~\cites{Bra1,Daw1}, but we only need the following simple principle.

Assume that $T_{\omega}$ is completely bounded with respect to the reduced C$^*$-norm on $\Pol(\G)$. The reduced operator space norm on $\Pol(\G)$ is a restriction of the operator space norm on $\vN^\infty(\G)$,
which is equal to the norm as the linear operator space dual of $\vN^\infty(\G)_*$. Using that $\Pol(\G)$ is $\sigma$-weakly dense in $\vN^{\infty}(\G)$, and using that functionals of the form $x \mapsto
\varphi(yx)$ for $y\in \Pol(\G)$ are norm-dense in $\vN^{\infty}(\G)_*$,
one infers that $\theta\circ T_{\omega}$ extends to a normal functional on $\vN^{\infty}(\G)$ for each $\theta\in \vN^{\infty}(\G)_*$, and that the formula
\[
M_{\omega}\colon \theta \mapsto \theta\circ T_{\omega}
\]
defines then  a completely bounded transformation on $\vN^\infty(\G)_*$ with the cb-norm bounded by that of $T_\omega$.
Taking the adjoint map of $M_\omega$, we obtain a completely bounded normal map on $\vN^\infty(\G)$, which extends $T_\omega$ with the same cb-norm.

When $\widehat{\qH}$ is a discrete quantum subgroup of $\widehat{\G}$, the central functionals on $\Pol(\G)$ restrict to central functionals on $\Pol(\qH)$, with possibly better cb-norm bounds.

When $\G_{1}$ and $\G_{2}$ are monoidally equivalent, there is a canonical one-to-one correspondence between their irreducible representation classes. This determines a canonical bijective correspondence between central linear functionals by means of the identification \eqref{EqId}. In
terms
of multipliers, this means that if $\omega$ is a central linear functional on $\Pol(\G_{1})$, the operator $T_{\omega}$ can be seen as acting on $\Pol(\G_{2})$ by the same formula $\sum_{\pi} \omega_{\pi} p_{\pi}$ . We still denote by
$T_{\omega}$ this new operator on $\Pol(\G_2)$.

\begin{Prop}[\cite{Fre1}*{Proposition 6.3}]\label{PropMon}
With the above notations, the completely bounded norms of $T_{\omega}$ with respect to the reduced C$^*$-algebras $C_r(\G_{1})$ and $C_r(\G_{2})$ coincide. In particular, $T_{\omega}$
is unital and completely positive on $C_r(\G_{1})$ if and only if it is unital and completely positive on $C_r(\G_{2})$.
\end{Prop}

\begin{Def}\label{DefACPAP}
A discrete quantum group $\widehat{\G}$ is said to have the \emph{central almost completely positive approximation property} (central ACPAP for short) if there is a net of central
functionals $(\psi_t)_{t\in I}$ on $\Pol(\G)$ such that
\begin{enumerate}
\item\label{eq:ACPAP-cpness}
For any $t$, the operator $\Psi_t = T_{\psi_t}$ on $\Pol(\G)$ induces a unital completely positive (ucp) map on $C_r(\G)$.
\item\label{eq:ACPAP-finapprox}
For any $t$, the operator $\Psi_t$ is approximated in the cb-norm by finitely supported central multipliers.
\item\label{eq:ACPAP-approxid}
For any $\pi \in \Irr(\G)$, $\lim_{t\in I}(\psi_t)_\pi =1$.
\end{enumerate}
\end{Def}

Proposition~\ref{PropMon} implies that this property is preserved under monoidal equivalence.

\begin{Rem}
When $\G$ is of Kac-type, the ACPAP condition (without requiring the centrality assumption) is equivalent to the central ACPAP, as the Kac case allows one to make an averaging process
by a conditional expectation~\cite{Kra1}*{Theorems~5.14 and~5.15}.\end{Rem}

\begin{Rem}
 The terminology above is motivated by the notion of completely positive approximation property: for coamenable Kac-type quantum groups such as $\SU_{\pm 1}(2)$, there is a net of finitely supported and completely positive central multipliers approximating the identity, which could be called the central completely positive approximation property.  Such approximating multipliers are a direct analogue of F{\o}lner sets for amenable discrete groups. However, even for the $\SU_{q}(2)$, one can no longer hope to produce such central completely positive finite rank multipliers simply out of the coamenability, precisely because of the monoidal equivalence with $\Op_F$, which is \emph{not} coamenable! 
\end{Rem}

\begin{Rem}
Since the first condition implies $(\psi_t)_1 = 1$ for the component of the trivial representation, the maps $\Psi_t$ of the above net extend to normal $\varphi$-preserving ucp maps on $\vN(\widehat{\G})$ which approximate the identity in the pointwise convergence with respect to the $\sigma$-weak topology on $\vN(\widehat{\G})$. Moreover, condition~\eqref{eq:ACPAP-finapprox} implies that $(\psi_t)_\pi$ is a $c_0$-sequence for any $t$. In particular, each $\Psi_t$ induces a compact operator on $L^2(\G, \varphi)$ and we obtain the Haagerup property for $(\vN(\widehat{\G}), \varphi)$.  Here, given a von Neumann algebra $M$ and a faithful
normal state $\psi$, $(M, \psi)$ is said to have the Haagerup property if there is a net of $\psi$-preserving normal completely positive maps $(\Psi_t)_{t \in I}$ on $M$ such that
the $\Psi_t$ converge to the identity in the topology of pointwise $\sigma$-weak convergence, and such that the $\Psi_t$ induce compact operators on $L^2(M, \psi)$ by the GNS-map of $\psi$ for each $t$.
\end{Rem}

\begin{Rem}
On the other hand, the representation of $\vN(\widehat{\G})$ on $L^2(\G, \varphi)$ is in standard form, and for any $t$ and $\xi \in L^2(\G, \varphi)$ the formula $x \mapsto \Psi_t(x)
\xi$ defines a compact map from $\vN(\widehat{\G})$ to $L^2(\G, \varphi)$. It follows that $\vN(\widehat{\G})$ has another variant of the Haagerup property in the sense of~\cite{Hou1}.
\end{Rem}

\begin{Prop}
The central ACPAP implies the CCAP for discrete quantum groups.
\end{Prop}

\begin{proof}
Suppose that $\hat{\G}$ has the central ACPAP, with a net $(\psi_t)_{t \in I}$ as in Definition~\ref{DefACPAP}.  Consider the product net $J = I \times \N$. For each $(t, n)$ we choose a finitely supported central multiplier $\Psi_{t, n}$ satisfying $\norm{\Psi_{t, n} - \Psi_t} < 2^{-n}$.  Putting $\tilde{\Psi}_{t, n} = (1 + 2^{-n})^{-1} \Psi_{t, n}$, we obtain a net $(\tilde{\Psi}_{t, n})_{(t, n) \in J}$ of finite rank complete contractions (with respect to the reduced norm) such that $\lim_{(t, n) \in J} (\tilde{\Psi}_{t, n})_\pi = 1$ for any $\pi \in \Irr(\G)$.  Then the net $(\tilde{\Psi}_{t, n})_{(t, n)}$ converges to the identity on $\Pol(\G)$ in the pointwise convergence topology with respect to the reduced norm.  By the density of $\Pol(\G)$ inside $C_r(\G)$ and the uniform bound on the norms of the $\tilde{\Psi}_{t, n}$, we obtain the same convergence on $C_r(\G)$.
\end{proof}

\begin{Rem}
In the above construction we may modify $\tilde{\Psi}_{t, n}$ by replacing the component of trivial representation $(\tilde{\Psi}_{t, n})_1$ with $1$.  This way we obtain a net of normal $\varphi$-preserving maps on $\vN^{\infty}(\G)$ which approximate the identity in the $\sigma$-weak pointwise convergence topology and satisfy $\limsup_{t,n} \bigl\|\tilde{\Psi}_{t, n}\bigr\|_\cb = 1$. Thus, the von Neumann algebra $\vN^{\infty}(\G)$ has the W$^*$-CCAP.
\end{Rem}

\section{Central approximation for free orthogonal quantum groups}
\label{sec:Aof-approx}

We now turn to our main results. Recall that the $\mu_d$ are the dilated Chebyshev polynomials. Motivated by the Kac-type case of \cite{Bra1}, we will show that, for $-1 < t < 1$, the multiplier operators $\Psi_t$ on $\Pol(\SU_q(2))$ given by
\begin{align}\label{eq:multiplier-defn}
\Psi_t &= \sum_{d \in \N} b_d(t) p_d,&
 b_d(t) &= \frac{\mu_{d}(\absv{q}^{t} + \absv{q}^{-t})}{\mu_{d}(\absv{q} + \absv{q}^{-1})}, \quad
\end{align}
extend to completely positive multipliers on $C(\SU_q(2))$ leading to the central ACPAP (here, as in Introduction, $p_d$ denotes the projection onto the isotypic subspace of spin $d/2$).

\begin{Prop}\label{Prop:ct-pos-state-SUq2}
For any $q \in (-1, 1) \setminus 0$ and $-1 < t < 1$, there is a central state $\psi_t$ on $C(\SU_{q}(2))$ such that $(\id \otimes \psi_t) \Delta$ equals $\Psi_t$.
\end{Prop}

\begin{proof}
First, since $b_d(t)$ is an even function in $t$, we may assume $0 \leqslant t < 1$.  We consider a representation $\pi_z$ of $C(\SU_q(2))$ constructed\footnote{One should note that this parametrization of $\pi_z$ is different from the one of~\cite{Voi1} by $z \leftrightarrow 1 - z$.} in~\cite{Voi1}*{Section 4}. Let
$\Pol(\SU_{q}(2)/\T)\subseteq \Pol(\SU_q(2))$ denote the linear span of the $u_{i0}^{(d)}$
for $d \in \N$ and $-d/2 \leqslant i \leqslant d/2$.

For every $z \in \C$, there is an action $\pi_z$ of $\Pol(\SU_{q}(2))$ on $\Pol(\SU_{q}(2)/\T)$ defined by
\[
\pi_{z}(x)y = f_{1 - z}(S(x_{(2)})) x_{(1)} y S(x_{(3)}).
\]
For any sequence of positive numbers $(M_k)_{k \in \N}$, we can define a new inner product 
$\langle \cdot, \cdot \rangle_M$ on $\Pol(\SU_q(2)/\T)$ by
\[
\left\langle u_{j0}^{(l)}, u_{i0}^{(k)} \right\rangle_M = \frac{\delta_{kl}\delta_{ij} M_k}{Q_{k, i} \dim_q(\pi)} \quad (k, l \in \frac{1}{2}\N),
\]
where $Q_{k, i}$ is the component of the Woronowicz character.  When $M_k = 1$, this is the usual inner product induced by $\varphi$.

For any $t \in \lbrack 0, 1\rbrack$, with the choice $M^{(t)}_k = \prod_{l \leq k} m(1-t, l)$ in the notation of~\cite{Voi1}*{Section~4}, the content of \cite{Voi1}*{Lemma~4.3} could be equivalently stated as that $\pi_t$ becomes a $*$-representation with respect to $\langle \, \cdot\,,\,\cdot\,\rangle_{M^{(t)}}$. It is necessarily bounded as $\Pol(SU_q(2))$ is generated by the components of a unitary matrix,
hence it extends to a $*$-representation $\omega_t$ of $C(\SU_q(2))$ on the Hilbert space completion $\Hsp_t$ of $\Pol(\SU_{q}(2)/\T)$ with respect to $\langle \,
\cdot\,,\,\cdot\,\rangle_{M^{(t)}}$ (recall that $C(\SU_q(2))$ is as well the universal C$^*$-envelope of the $*$-algebra $\Pol(\SU_q(2))$).

Let $\psi_{t}$ be the vector state $\psi_t(x) = \langle  \omega_t(x) u^{(0)}_{0 0},  u^{(0)}_{0 0} \rangle_{M^{(t)}}$ on $C(\SU_{q}(2))$ associated with the element $u^{(0)}_{0 0} \in \Pol(\SU_{q}(2)/\T)$. As $M^{(t)}_0 = 1$, we have
\begin{equation*}
\psi_{t}(x) = f_{1 - t}(S(x_{(2)})) \varphi(x_{(1)} S(x_{(3)})).
\end{equation*}
Using the basic identities from Section~\ref{sec:prelim}, we compute
\begin{align*}
\psi_{t}(u^{(d/2)}_{i j}) & = \sum_{k, l}f_{1 - t} \Bigl(S\Bigl(u^{(d/2)}_{k l}\Bigr)\Bigr) \varphi\Bigl(u^{(d/2)}_{i k} S\Bigl(u^{(d/2)}_{l j}\Bigr)\Bigr) \\
& = \sum_{k} \delta_{i j} Q_{d/2, k}^{t - 1} \frac{Q_{d/2, k}}{\dim_{q}(U_{d/2})} && (\text{by~\eqref{eq:inn-prod}})\\
& = \frac{\mu_d(\absv{q}^t + \absv{q}^{-t})}{\mu_d(\absv{q} + \absv{q}^{-1})} \delta_{i j} && (\text{by~\eqref{eq:SUq2-q-dim-formula}}).
\end{align*}
Comparing the above with~\eqref{eq:multiplier-defn}, we obtain the assertion.
\end{proof}

The next theorem can be interpreted as the uniform boundedness of the representations $\pi_z$ for $z \in \D$, the strip formed by complex numbers with real part between $-1$ and $1$.

\begin{Theorem}\label{Thm:Psi-z-slice-by-bdd-psi-z}
For any $q \in (-1, 1) \setminus 0$, there is a holomorphic map $z \mapsto \theta_z$ from $\D$ to the space of bounded central functionals on $C(\SU_{q}(2))$, such that the induced
completely bounded map $\Theta_{z} = (\theta_{z} \otimes \imath)\circ\Delta$ is a nonzero positive scalar multiple of $\Psi_t$ when $z = t\in (-1, 1)$.
\end{Theorem}

In order to prove this theorem, we investigate the representations $\omega_t$ which appeared in the proof of Proposition~\ref{Prop:ct-pos-state-SUq2} in more detail in a series of
lemmas. In the sequel, we will keep using the notations of the proof of Proposition \ref{Prop:ct-pos-state-SUq2}.

Recall the representation $\rho_q$ of $C(SU_q(2))$ on $\ell^2(\N)\otimes \ell^2(\Z)$. Let $V$ denote the isometry $\ell^2(\N) \to \ell^2(\N) \otimes \ell^2(\Z)$ given by $\xi \to \xi \otimes e_0$.  Then the formula $E(T) = (V^* T V) \otimes 1$ defines a conditional expectation from $B(\ell^2(\N) \otimes \ell^2(\Z))$ onto $B(\ell^2(\N)) \otimes 1$.  The restriction of $E$ to $\rho_q(C(\SU_q(2))$ can be considered as a conditional expectation from $C(\SU_q(2))$ onto $C^*(\alpha)$.  More conceptually, it can be defined as the averaging with respect to the scaling group of $\SU_q(2)$ generated by the square of the antipode.  Concretely, $E$ on $C(\SU_q(2))$ can be expressed as $E(u_{ij}^{(d/2)}) = \delta_{ij}u_{ii}^{(d/2)}$. From this description, we
obtain the following lemma.

\begin{Lem}\label{Lem:central-state-compress-to-alpha}
Let $\psi$ be a central state on $C(\SU_q(2))$. We have $\psi(x) = \psi(E(x))$.
\end{Lem}

\begin{Lem}\label{LemEigen}
For any $-1 < t < 1$, we have $\omega_t(q^{-1} \alpha + q \alpha^*) u_{00}^{(0)} = \sgn(q) (\absv{q}^{t} + \absv{q}^{-t}) u_{00}^{(0)}$.
\end{Lem}
\begin{proof} This follows from a direct computation, using the definition of $\omega_t$ and the formulas in Section \ref{sec:prelim}.
\end{proof}

We let $\Ksp_t$ denote the closed span of $\omega_t(C^*(\alpha))u^{(0)}_{0 0}$ in $\Hsp_t$. We denote the representation of $C^*(\alpha)$ on $\Ksp_t$ induced by $\omega_t$ as $\omega^0_t$.

\begin{Lem}\label{LemEquiv}
For any $-1 < t<1$, the $*$-representation $\omega^0_t$ of $C^*(\alpha)$ is unitarily equivalent to $\rho^0_q$ (cf.~Subsection \ref{SecQuantumSU}).
\end{Lem}

\begin{proof}
It is well-known that any irreducible representation of the Toeplitz algebra is either a character or is isomorphic to $\rho^0_q$. Hence $\Hsp_t$ splits into a direct sum of
$C^*(\alpha)$-representations $\Hsp_t^{(0)}\oplus \Hsp_t^{(1)}$, where $\Hsp_t^{(0)}$ is an amplification of $\rho^0_q$ and $\Hsp_t^{(1)}$ consists of a direct integral of characters.
By Lemma \ref{LemEigen}, it follows immediately that $u_{00}^{(0)} \in \Hsp_t^{(0)}$. Thus, $\omega^0_t$ is an amplification of $\rho^0_q$. Moreover, as $\Ksp_t$ is contained in the
closed linear span of the $u_{00}^{(d)}$, the tridiagonal form of $\omega_t(\alpha)$ on that span (cf.~\cite{Voi1}) implies that $\ker(\omega^0_t(\alpha))$ is at most one-dimensional.
Consequently, $\omega_t^0$ is equivalent to $\rho_q^0$ itself.
\end{proof}

Before embarking on the proof of Theorem~\ref{Thm:Psi-z-slice-by-bdd-psi-z}, let us introduce some notations from the theory of $q$-special functions (see e.g.~\cite{Koe1} for details).
Given parameters $q$, $x$, and a nonnegative integer $k$, we set
\[
(x; q)_k = (1 - x) (1 - x q) \cdots (1 - x q^{k - 1}),
\]
with the convention $(x; q)_0 = 1$. For any pair of nonnegative integers $k \leqslant n$, the \emph{$q$-binomial coefficient} is given by
\[
\qbin{n}{k}{q} = \frac{(q; q)_n}{(q; q)_k (q; q)_{n - k}}.
\]
With $x = y + y^{-1}$, the \emph{continuous $q$-Hermite polynomials} (in dilated form) are given by
\[
H_n(x; q) = \sum_{k = 0}^n \qbin{n}{k}{q} y^{ (n - 2 k)}.
\]
This corresponds to $H_n((x/2) \mathinner{|} q)$ in the notation of~\cite{Koe1}*{Section~3.26}.  They satisfy the recurrence relation
\begin{equation}\label{eq:q-Hermite-recurr-formula}
x H_n(x; q) = H_{n + 1}(x; q) + (1 - q^n) H_{n - 1}(x; q).
\end{equation}

\begin{proof}[Proof of Theorem~\ref{Thm:Psi-z-slice-by-bdd-psi-z}]
For the sake of simplicity we assume that $q$ is positive. The general case be handled in the same way by simply replacing all occurrences of $q$ by $\absv{q}$.

For each $-1 < t < 1$, consider
\begin{equation}\label{eq:defn-coeff-of-eta-t}
p_n(t) = \frac{q^n}{\sqrt{(q^2;q^2)_n}} H_n(q^t + q^{-t}; q^2) = \frac{q^n}{\sqrt{(q^2;q^2)_n}} \sum_{k = 0}^n \qbin{n}{k}{q^2} q^{t (n - 2 k)}.
\end{equation}
Since the sequence $(q^2; q^2)_k$ monotonically decreases to some nonzero number $D$ as $k \to \infty$, the right hand side of the above can be bounded by $D^{-5/2} q^{(1 + t) n} \sum_{k=0}^n q^{-2 t k} < D' q^{(1 + t) n - 2 n t}$ for another constant $D'$.  Since one has $t < 1$, the sequence $(p_n(t))_{n \in \N}$ is square summable.  Moreover, \eqref{eq:q-Hermite-recurr-formula} implies that
\begin{align*}
(q^t + q^{-t}) p_n(t) &= \frac{(q^t + q^{-t}) q^n}{\sqrt{(1 - q^2) \cdots (1 - q^{2 n})}} H_n(q^t + q^{-t}; q^2) \\
&= \frac{q^n \sqrt{1 - q^{2 (n + 1)}}}{\sqrt{(q^2; q^2)_{n + 1}}} H_{n + 1}(q^t + q^{-t}; q^2) + \frac{q^n \sqrt{1 - q^{2 n}}}{\sqrt{(q^2; q^2)_{n - 1}}} H_{n - 1}(q^t + q^{-t};
q^2)\\
&= q^{-1} \sqrt{1 - q^{2 (n + 1)}} p_{n + 1}(t) + q \sqrt{1 - q^{2 n}} p_{n - 1}(t).
\end{align*}
This implies that the vector $\eta_t = \sum p_n(t) e_n \in \ell^2(\N)$ is an eigenvector of $\rho_q^0(q^{-1} \alpha + q \alpha^*)$ with eigenvalue
$q^t + q^{-t}$.  Lemmas~\ref{LemEigen} and~\ref{LemEquiv} imply that there exists $C_t>0$ such that
\[
C_t\psi_t(x) = \langle \rho_q^0(E(x))\eta_t,\eta_t\rangle,\qquad \forall x\in C(SU_q(2)).
\]

Note that the right-hand side of~\eqref{eq:defn-coeff-of-eta-t} is an expression in $t$ which is positive for real $t$. If we replace $t$ by $z=t + i s$ for $s \in \R$, each
term $q^{z (n - 2 k)}$ has modulus not greater than $q^{t (n - 2 k)}$. Since we already know the $\ell^2$-convergence of $(p_n(t))_{n \in \N}$ for $-1 < t < 1$, we obtain the
$\ell^2$-convergence of $(p_n(z))_{n \in \N}$ for arbitrary $-1 < \Re(z) < 1$.  Now, the vectors $\eta_z = \sum_n p_n(z) e_n \in \ell^2(\N)$ is holomorphic in the variable $z$.  For example, one can directly see this from the estimate
$$
\left | \partial_z p_n(z) \right | \leqslant  D^{-\frac{5}{2}}\sum_{k=0}^n \left | n - 2 k \right | q^{(1 + \Re(z)) n - 2 \Re(z) k}
$$
for $n \in \N$.  The right hand side above is yet again bounded by $D^{-5/2} n^2 q^{(1-\Re(z)) n}$, which is certainly square summable in $n$.  It follows that the bounded functionals $\theta_z$ on $C^*(\alpha)$ defined by $\theta_z(x) = \langle
 \rho^0_q(x) \eta_z ,\eta_{\overline{z}}\rangle$ depend holomorphically on $z$.  By composing these with $E$, we obtain a
holomorphic family of functionals on $C(\SU_q(2))$, which is again denoted by $(\theta_z)_{z \in \D}$.  Setting $C_z = \sum_n p_n(z)^2$, we have the equality
\[
\theta_z\bigl(u^{(d/2)}_{i j}\bigr) = C_z\delta_{ij}\frac{\mu_d(q^{z}+q^{-z})}{\mu_d(q + q^{-1})}.
\]
for $z=t \in (-1, 1)$. Since both sides are holomorphic in $z$ and agree on the interval $(-1, 1)$, we obtain the equality for all $-1 < \Re(z) < 1$.  In particular, $\theta_z$ is
central.  Finally, when $t \in (-1, 1)$,
the number $C_t$ is a strictly positive real number by the expression~\eqref{eq:defn-coeff-of-eta-t}.
\end{proof}

The next two lemmas allow us to connect Theorem~\ref{Thm:Psi-z-slice-by-bdd-psi-z} with the central ACPAP of free orthogonal groups.

\begin{Lem}\label{Lem:sp-proj-cb-norm}
The completely bounded norm of $p_d$ is bounded from above by $\mu_d(|q + q^{-1}|)^2$.
\end{Lem}

\begin{proof}
In general, for any compact quantum group $\G$ and any irreducible representation $\pi$, formula~\eqref{eq:proj-pi-as-fourier-transform} implies that the cb-norm of $p_\pi$ is bounded by the norm of the
quantum
character $\dim_q(\pi) \sum_i Q_{\pi, i}^{-1} u^{(\pi) *}_{i i}$.
One sees that the latter quantity is bounded by $\dim_q(\pi)^2$ using $\Tr(Q_\pi^{-1}) = \Tr(Q_\pi) = \dim_q(\pi)$. Since we have $\dim_q(\pi_{d / 2}) = \mu_d(|q+q^{-1}|)$, we obtain
the
assertion.
\end{proof}

\begin{Lem}[cf.~\cite{Bra1}*{Proposition~4.4}]\label{Lem:Chebyshev-growth}
Let $q$ be a real number satisfying $0 < q < 1$, and let $z$ be a complex number satisfying $0 <\Re(z) < 1$. We have the estimate
\[
\absv{\frac{\mu_d(q^z + q^{-z})}{\mu_d(q + q^{-1})}} = O(q^{d (1 - \Re(z))}) \quad (d \to \infty).
\]
\end{Lem}

\begin{proof}
From the recurrence relation of the Chebyshev polynomials, we see that
\begin{equation}\label{Eq:Chebyshev-growth}
\mu_d(q^z + q^{-z}) = q^{d z} + q^{(d - 2) z} + \cdots + q^{- d z} = \frac{q^{(d + 1) z} - q^{-(d + 1) z}}{q^z - q^{-z}}
\end{equation}
for any complex number $z$. By $\Re(z) > 0$, one has $q^{d z} \rightarrow 0$ while $q^{-d z} \rightarrow \infty$ as $d \rightarrow \infty$. Thus, we have the estimate
\begin{align*}
\absv{\frac{\mu_d(q^z + q^{-z})}{\mu_d(q + q^{-1})}}
&= \absv{ \frac{q^{(d+1)z} - q^{-(d+1)z}}{q^z - q^{-z}} } \cdot \absv{ \frac{q - q^{-1}}{q^{d+1} -
q^{-(d+1)}} }\\
& \sim \absv{ \frac{q - q^{-1}}{q^z - q^{-z}} } \cdot \absv{ \frac{q^{-(d+1)z}}{q^{-(d+1)}}}
= O(q^{d(1-\Re(z))}).
\end{align*}
This proves the assertion.
\end{proof}

\begin{Theorem}\label{Thm:AoF-central-approx}
Let $F$ be an invertible matrix of size $N$ satisfying $F \overline{F} \in \R I_{N}$. The discrete quantum group $\FO_F$ has the central ACPAP.
\end{Theorem}

\begin{proof}
Consider the multiplier operators $\Psi_{t} = \sum_d b_d(t) p_{d/2}$ on $\Pol(\Op_F)$ given by the sequence $(b_{d}(t))_{d}$ as in~\eqref{eq:multiplier-defn}.  We claim that the operators
$\Psi_t^3 = \sum_d b_d(t)^3 p_{d/2}$ for $0 < t < 1$ give a desired net of central multipliers approximating the identity as $t \rightarrow 1$. First, by definition,
condition~\eqref{eq:ACPAP-approxid} of Definition \ref{DefACPAP} is readily satisfied. Let $q \in (-1, 1)$ be such that $\SU_q(2)$ is monoidally equivalent to $\Op_F$ (the case $q=\pm 1$ being obvious, cf.~the
Introduction).
Proposition~\ref{PropMon} implies that it
is enough to check conditions~\eqref{eq:ACPAP-cpness} and~\eqref{eq:ACPAP-finapprox} for $\Psi_t^3$ acting on $C_r(\SU_q(2))$.

Proposition~\ref{Prop:ct-pos-state-SUq2} implies condition~\eqref{eq:ACPAP-cpness}. It remains to prove condition~\eqref{eq:ACPAP-finapprox}. Lemmas~\ref{Lem:sp-proj-cb-norm}
and~\ref{Lem:Chebyshev-growth} imply that $(b_d(t)^3 \norm{p_d}_\cb)_{d \in \N}$ is absolutely summable for $0 < t < 0.3$. In particular, $\Psi_t^3$ is approximated by
finitely supported central multipliers in the cb-norm when $t$ is in this range.

It follows that the holomorphic map $\Theta_z^3$ from $\D$ to the space $M$ of completely bounded multipliers given by
Theorem~\ref{Thm:Psi-z-slice-by-bdd-psi-z} remains in the cb-norm closure $M^c_0$ of finitely supported central multipliers when $z \in (0, 0.3)$. Since any nontrivial holomorphic function
in one variable has isolated zeros, $\Theta_z^3$ regarded as a map into the Banach space $M / M^c_0$ must be trivial. Thus, we obtain $\Psi_t^3 \in M^c_0$ for arbitrary $0< t < 1$.
\end{proof}

\begin{Rem}
When $0 < \absv{q} \leqslant 1 / \sqrt{3}$, the cb-norm of the projections $p_d$ are actually dominated by a quadratic polynomial in $d$~\cite{Fre1}*{Theorem~5.4}. Combining this
with the exponential decay of $b_d(t)$ for $0 \leqslant t < 1$, one obtains another proof of Theorem~\ref{Thm:AoF-central-approx} for the matrices $F \in M_N(\C)$ of size $N > 2$
satisfying $F \overline{F} \in \R I_N$, directly from Proposition~\ref{Prop:ct-pos-state-SUq2}.  However, we need the full generality of Theorem~\ref{Thm:AoF-central-approx} for the later application in Theorem~\ref{Thm:free-q-cent-Hage}.
\end{Rem}

\begin{Rem}
A well known strategy to construct a net of $c_0$ completely positive multipliers approximating the identity is to use the L\'{e}vy process $(e^{- t \psi})_{t > 0}$ associated with a
conditionally positive functional $\psi$ of large growth. L\'{e}vy processes on $\SU_q(2)$ were studied by Sch\"{u}rmann and Skeide in the 90's~(\cite{Ske1}, etc.). However, it
seems difficult to find a conditionally positive functional with both the centrality and large growth directly from their presentation. The centrality can be achieved by
``averaging'' with respect to the adjoint action, but for quantum groups of non-Kac-type, this operation can affect the growth estimate and condition~\eqref{it:summability-condi} in
the Introduction may
not hold anymore after such an operation~\cite{Cip1}.
\end{Rem}

\section{Complements on universal quantum groups}
\label{Sec:compl-free-qgrp}

\subsection{Free products}

Let us briefly review the free product construction of quantum groups, mainly from~\cite{Wan1}.  Let $\G_1$ and $\G_2$ be two compact quantum groups, and write $\Pol(\G_i)^\circ = \oplus_{\pi \in \Irr(\G) \setminus 1} \Pol(\G_i)_\pi$ for the orthogonal complement of the unit with respect to the Haar state.  Then the free product of the unital $*$-algebras $\Pol(\G_1)$ and $\Pol(\G_2)$ (with respect to the Haar states) can be realized as the direct sum of $\C$ and all the possible alternating tensor products of $\Pol(\G_1)^\circ$ and $\Pol(\G_2)^\circ$, i.e.~those whose adjacent tensor factors are distinct. It carries a natural structure of Hopf $*$-algebra, containing both $\Pol(\G_1)$ and $\Pol(\G_2)$ as Hopf $*$-subalgebras. In the following we write $\Pol(\G_1) \frprd_r \Pol(\G_2)$ for this free product. As the Haar state on this free product is the free product of the Haar states, the reduced C$^*$-algebra $C_r(\G)$ associated to $\Pol(\G)=\Pol(\G_1)\frprd_r \Pol(\G_2)$ coincides with the C$^*$-algebraic reduced free product $(C_r(\G_1),\varphi_1)\rfrprd (C_r(\G_2),\varphi_2)$.

The irreducible unitary corepresentations of $\Pol(\G_1) \frprd_r \Pol(\G_2)$ are given by
\begin{equation}\label{eq:irrep-free-prod}
\pi_1 \otimes \pi_2 \otimes \cdots \otimes \pi_n \quad \left ( \pi_j \in ( \Irr(\G_1) \setminus 1 ) \coprod ( \Irr(\G_2) \setminus 1 ) \right )
\end{equation}
for $n \in \N$, such that if $\pi_j$ is in $\Irr(\G_1)$ then $\pi_{j+1}$ is in $\Irr(\G_2)$ and vice versa.  The case $n = 0$ corresponds to the trivial representation, and in general the spectral subspace corresponding to~\eqref{eq:irrep-free-prod} is given by $\Pol(\G_{j_1})_{\pi_1} \otimes \cdots \otimes \Pol(\G_{j_n})_{\pi_n}$, where the $j_k \in \{1, 2\}$ satisfies $\pi_{j_k} \in \Irr(\G_{j_k})$.  The dual discrete quantum group represented by $\Pol(\G_1) \frprd_r \Pol(\G_2)$ is denoted by $\widehat{\G}_1 * \widehat{\G}_2$, and called the \emph{free product} of $\widehat{\G}_1$
and $\widehat{\G}_2$.

Monoidal equivalence is compatible with free products. Indeed, if compact quantum groups $\G_i$ and $\qH_i$ are monoidally equivalent for $i = 1, 2$, the universality of free
products implies that the free product $\Pol(\G_1, \qH_1) \frprd_r \Pol(\G_2, \qH_2)$ with respect to the invariant states admits a left coaction of $\Pol(\G_1) \rfrprd
\Pol(\G_2)$ and a right coaction of $\Pol(\qH_1) \rfrprd \Pol(\qH_2)$.   Analogous to the case of function algebra, the spectral subspace corresponding to~\eqref{eq:irrep-free-prod} is given by $\Pol(\G_{j_1}, \qH_{j_1})_{\pi_1} \otimes \cdots \otimes \Pol(\G_{j_n}, \qH_{j_1})_{\pi_n}$.  In particular, the coactions of the free product algebras are both ergodic.  This provides a linking algebra between the free product quantum groups.

\subsection{Free unitary quantum groups}

Let $N \geqslant 2$ be an integer, and $F \in M_N(\C)$ invertible. The \emph{free unitary quantum group} $\Up_F$ associated with $F$ is given by the Hopf $*$-algebra $\Pol(\Up_F)$ which is
universally generated by the elements $(u_{i j})_{1 \leqslant i, j \leqslant N}$ subject to the condition that the matrices $U = [u_{i j}]$ and $F \overline{U} F^{-1}$ are unitary in
$M_N(\Pol(\Up_F))$. The structure of $\Up_F$ does not change under transformations of the form $F \rightarrow \lambda V F W$ for $\lambda \in \C^\times$ and $V, W \in \U(N)$.

Banica showed the following properties of $\Pol(\Up_F)$~\cite{Ban6}:

\begin{enumerate}
\item
The irreducible unitary representations of $\Up_F$ are labeled by the vertices of the rooted tree $\N * \N$.
\item
When $F \overline{F} \in \R I_N$, there is an inclusion of discrete quantum groups $\FU_F \rightarrow \Z * \FO_F$ given by the Hopf $*$-algebra homomorphism
\[
\Pol(\Up_F) \rightarrow \Pol(\U(1)) \rfrprd \Pol(\Op_F), \quad u_{i j} \mapsto z \cdot u_{i j}.
\]
Here, $z$ denotes the coordinate function on $\U(1) \subset \C$.
\end{enumerate}

In fact, a free unitary quantum group $\Up_F$ is always monoidally equivalent to another one, $\Up_Q$, with $Q \in M_2(\C)$ satisfying $Q \overline{Q} \in \R I_2$. Rescaling $F$ so that it
satisfies $\Tr(F F^*) = \Tr((F F^*)^{-1})$, the $Q$ can be taken as the off-diagonal matrix with coefficients $q^{\frac{1}{2}}$ and $q^{-\frac{1}{2}}$ characterized by $q + q^{-1} = \Tr(F
F^*)$~\cite{BDV1}*{Section~6}.

For a general matrix $F \in M_N(\C)$, Wang showed the following structural results~\cite{Wan3}: the free orthogonal quantum group $\FO_F$ admits a presentation as a free
product
\begin{equation}\label{eq:gen-free-orth-free-prod}
\FO_F \simeq \FU_{P_1} * \cdots * \FU_{P_m} * \FO_{Q_1} * \cdots * \FO_{Q_n},
\end{equation}
where $P_i$ and $Q_j$ are appropriate matrices of smaller size satisfying $Q_i \overline{Q}_i \in \R I_{N_i}$.

\subsection{Quantum automorphism groups}

Let $(B, \psi)$ be a pair consisting of a finite-dimensional C$^*$-algebra and a state on it. The \emph{compact quantum automorphism group} of $(B, \psi)$~\cites{Wan2,Ban5} is defined
as the
universal Hopf $*$-algebra $\Pol(\qAut(B, \psi))$ which can be endowed with a coaction
\[
\beta\colon B \rightarrow B \otimes \Pol(\qAut(B, \psi))
\]
such that $(\psi \otimes \iota) \circ \beta (x) = \psi(x) 1$ for $x \in B$.

Let $m\colon B \otimes B \rightarrow B$ be the product map. A state $\psi$ is said to be a $\delta$-form if $m \circ m^*$ is equal to $\delta \id_B$ for some $\delta > 0$, where
$m^*$ is the adjoint of $m$ with respect to the GNS-inner products associated with the states $\psi \otimes \psi$ on $B \otimes B$ and $\psi$ on $B$~\cite{Ban5}.

\begin{Prop}\label{PropQAutIsSOq3}
Assume that $\psi$ is a $\delta$-form on a finite-dimensional C$^*$-algebra $B$. Then, the quantum group $\qAut(B, \psi)$ is monoidally equivalent to $\SO_q(3)$ for the $q$ satisfying
$(q +
q^{-1}) = \delta$.
\end{Prop}

\begin{proof}
This follows from the combination of \cite{DeR1}*{Section~9.3} and \cite{Sol1}.
\end{proof}

The following result is due to Brannan but did not appear in the litterature. We therefore include a proof for the sake of completeness.

\begin{Prop}[Brannan]\label{PropQauGenFreeProd}
Let $B$ be a finite-dimensional C$^*$-algebra and let $\psi$ be a state on it. Let $B = \oplus_{i = 1}^k B_i$ be the coarsest direct sum decomposition into C$^*$-algebras such that the normalization of $\psi |_{B_i}$ is a $\delta_i$-form for some $\delta_i$ for each summand. Then, $\widehat{\qAut}(B, \psi)$ is isomorphic to the free product $*_{i=1}^k
\widehat{\qAut}(B_i, \psi_i)$, where $\psi_i$ is the state on $B_i$ obtained by normalizing $\psi|_{B_i}$.
\end{Prop}

\begin{Rem} Note that the restricted normalization of $\psi$ to each simple block of $B$ automatically gives some $\delta$-form, hence the above direct sum decomposition exists.
\end{Rem}

\begin{proof}
For each $i$, put $H_i = L^2(B_i, \psi_i)$. Similarly, put $H = L^2(B, \psi)$. The assumption that $\oplus_i B_i$ is a coarsest decomposition as in the assertion implies that the eigenvalues of $m \circ m^*$ on $H_i$ are different for distinct
$i$. As $m$ is an intertwiner, it follows that $H$ decomposes into a direct sum of the $H_i$ as a
representation of $\qAut(B, \psi)$. For each $i$, the coaction $H_i \rightarrow H_i \otimes \Pol(\qAut(B, \psi))$ is given by an algebra homomorphism preserving the state $\psi_i$.
It follows that there is a Hopf $*$-algebra homomorphism $\Pol(\qAut(B_i, \psi_i)) \rightarrow \Pol(\qAut(B, \psi))$ covering this coaction. Then, the direct sum $H$ admits a
coaction $\beta'$ of $\Pol(\qAut(B_1, \psi_1)) \frprd_r \cdots \frprd_r \Pol(\qAut(B_k, \psi_k))$. By the universality of free product construction, one obtains a homomorphism
\[
 \Pol(\qAut(B_1, \psi_1)) \frprd_r \cdots \frprd_r \Pol(\qAut(B_k, \psi_k)) \rightarrow \Pol(\qAut(B, \psi)).
\]

Conversely, the coaction $\beta'$ is implemented by a $\psi$-preserving $*$-homomorphism. Thus, it induces a Hopf $*$-algebra homomorphism
\[
\Pol(\qAut(B, \psi)) \rightarrow \Pol(\qAut(B_1, \psi_1)) \frprd_r \cdots \frprd_r \Pol(\qAut(B_k, \psi_k)),
\]
which is inverse to the above homomorphism.
\end{proof}

\section{Approximation properties for free quantum groups}

Let $\G$ be a compact quantum group and let $\widehat{\qH}$ be a discrete quantum subgroup of $\widehat{\G}$. Since the complete boundedness and complete positivity of central
multiplier
operators on $\Pol(\G)$ pass to the subalgebra $\Pol(\qH)$, one has the following lemma.

\begin{Lem}\label{Lem:ACPAP-pass-to-subgr}
The central ACPAP is preserved under taking discrete quantum subgroups.
\end{Lem}

A little more involved is the permanence of under free products, which is shown by a direct adaptation of a result due to Ricard and Xu~\cite{Ric1}.

\begin{Prop}\label{Prop:ACPAP-stable-under-free-prod}
The central ACPAP is preserved under taking free products.
\end{Prop}

\begin{proof}
We follow the argument of~\cite{Ric1}*{Proposition~4.11}. Let $\widehat{\G}$ and $\widehat{\qH}$ be discrete quantum groups with the central ACPAP, and $(\Phi_i)_{i \in I}$,
$(\Psi_j)_{j \in J}$ be
corresponding nets of central ucp multipliers. Since each of the $\Phi_i$ (resp.\@ $\Psi_j$) is a unital multiplier on $\widehat{\G}$ (resp.\@ on $\widehat{\qH}$), it preserves the Haar state.  Thus, we obtain a free product ucp map $\Phi_i * \Psi_j$~\cite{Bla1} on $C_r(\G) \frprd_r C_r(\qH)$ For any pair $(i, j) \in I \times J$.

Let $\reg{A} = \Pol(\G) \frprd_r \Pol(\qH)$. For each $d \in \N$, let $\reg{A}_d$ be the homogeneous part of degree $d$ in $\reg{A}$ (that is, the direct sum of all the $d$-fold
alternating tensor products of $\Pol(\G)^\circ$ and $\Pol(\qH)^\circ$). Let $P_d$ be the projection from $\reg{A}$ onto $\reg{A}_d$. From~\cite{Ric1}*{Section~3}, we know that the map
$T_r = \sum r^d P_d$ is ucp for $0 \leqslant r < 1$, and that one has the estimate
\[
\norm{\sum_{d = n + 1}^\infty r^d P_d}_\cb \leqslant \frac{4 n r^n}{(1 - r)^2}.
\]
We claim that a subnet $((\Phi_i * \Psi_j) T_r)_{(i, j, r) \in I \times J \times \R_+}$ satisfies the desired conditions in Definition~\ref{DefACPAP}.

Since both of $(\Phi_i * \Psi_j)$ and $T_r$ are given by a central ucp multipliers of $\widehat{\G} * \widehat{\qH}$, so is their composition $(\Phi_i * \Psi_j) T_r$. Thus, it is enough to show that it can be approximated by finite rank
central
multipliers in the cb-norm. For simplicity, we write $\Phi = \Phi_i$, $\Psi = \Psi_i$, and fix an error tolerance $\delta > 0$.

Letting $r = 1 - \frac{1}{\sqrt{N}}$, we can find a large enough $N$ such that $(4 N r^N) (1 - r)^{-2} < \delta$. Then,
\[
E_N = (\Phi * \Psi) \left(\sum_{d = 0}^N r^d P_d\right)
\]
satisfies $\norm{(\Phi * \Psi) T_r - E_N}_\cb < \delta$.

For any $\epsilon > 0$, we can choose a finite rank central multiplier $\Phi^0$ on $\widehat{\G}$, which is $\epsilon$-close to $\Phi$ in the cb-norm. Then, the norms of $\Phi -
\Phi^0$
in $B(L^2(\Pol(\G), \varphi))$ and $B(L^2(\Pol(\G)^\opos, \varphi))$ are also bounded by $\epsilon$. Similarly, we choose a finite rank central multiplier $\Psi^0$ on $\widehat{\qH}$
with the same condition.  Applying~\cite{Ric1}*{Lemma~4.10} to $T_{1, k} = \Phi_0$ and $T_{2, k} = \Psi_0$ (independent of $k$), one sees that the cb-norm of $\Phi^0 * \Psi^0$ is bounded by $(2 d + 1) (1 + \epsilon)^d$ on $\reg{A}_d$. Then, $D_N = \sum_{d \leqslant N} r^d (\Phi^0 *
\Psi^0) P_d$ is finite-rank, completely bounded and is implemented by a central multiplier. The cb-norm of $E_N - D_N$ is bounded by
\[
\sum_{d=1}^N 4 d 2^d\epsilon
\]
by \cite{Ric1}*{proof of Proposition~4.11}. Thus, if we choose $\epsilon$ small enough (depending on $N$), we can make $\norm{(\Phi * \Psi) T_r - D_N}_\cb$ smaller than $2
\delta$. This proves the assertion.
\end{proof}

Using the above results, we can extend Theorem~\ref{Thm:AoF-central-approx} to the following universal quantum groups.

\begin{Theorem}\label{Thm:free-q-cent-Hage}
When $\G$ is one of the following compact quantum groups, $\widehat{\G}$ has the central ACPAP.
\begin{enumerate}
\item $\SU_{q}(2)$ or $\SO_{q}(3)$ for any $q\in [-1, 1]\setminus\{0\}$,
\item $\Op_F$ or $\Up_F$ for any matrix $F$,
\item $\qAut(B, \psi)$, for any finite-dimensional C$^*$-algebra $B$ and any state $\psi$.
\end{enumerate}
\end{Theorem}

\begin{proof}
The case of $\SU_q(2)$, and more generally of the $\Op_F$ with $F \overline{F} \in \R I_N$, are already proved in Theorem~\ref{Thm:AoF-central-approx}. By
Lemma~\ref{Lem:ACPAP-pass-to-subgr}, we obtain the case for $\SO_q(3)$.

Since $\Z = \widehat{\U(1)}$ is amenable, Proposition~\ref{Prop:ACPAP-stable-under-free-prod} implies that $\Z * \FO_F$ has the central ACPAP when $F$ satisfies $F
\overline{F} \in \R I_N$. Again by Lemma~\ref{Lem:ACPAP-pass-to-subgr}, we obtain the central ACPAP for $\FU_F$ with such an $F$.

For a general $\FU_F$, the central ACPAP follows from the monoidal equivalence with $\FU_G$ for a suitable chosen $G$ satisfying $G \overline{G} \in \R I_N$. Then,
Proposition~\ref{Prop:ACPAP-stable-under-free-prod} and the decomposition~\eqref{eq:gen-free-orth-free-prod} implies the case for a general $\FO_F$.

Similarly, the central ACPAP for $\qAut(B, \psi)$ when $\psi$ is a $\delta$-form follows from the monoidal equivalence with $\SO_q(3)$. Then, the general case follows from
Propositions~\ref{PropQauGenFreeProd} and~\ref{Prop:ACPAP-stable-under-free-prod}.
\end{proof}

Consequently, for the above quantum groups, the reduced C$^*$-algebra $C^*_r(\widehat{\G})$ and the von Neumann algebra $\vN(\widehat{\G})$ have the Haagerup property relative to the Haar state, and the (W*)CCAP.  Note that this was known as for the quantum groups above which are monoidally equivalent to a quantum group of Kac-type, by \cite{Bra1}, \cite{Bra2} and \cite{Fre1}. We can now answer in the affirmative the question on (W$^*$)CCAP raised at the end of Section~1 in~\cite{Iso1}. Recall that a \emph{Cartan subalgebra} in a von Neumann subalgebra is a maximal commutative
subalgebra whose normalizer generates the whole von Neumann algebra.  Now, Theorem~\ref{Thm:free-q-cent-Hage} and the results of~\cite{Iso2} implies the following structure result.

\begin{Prop}[\cite{Iso2}*{Corollary~C}]
Any of the following von Neumann algebras has no Cartan subalgebras if it is noninjective.
\begin{enumerate}
\item
$\vN(\FU_F)$, for any matrix invertible $F \in M_N(\C)$,
\item
$\vN(\FO_F)$, for $F$ satisfying $F \overline{F} \in \R I_N$,
\item
$\vN^\infty(\qAut(B, \psi))$, for a $\delta$-form $\psi$ on a finite-dimensional C$^*$-algebra $B$.
\end{enumerate}
\end{Prop}

We note that various related structural results on these von Neumann algebras are previously known; in particular, the (genaralized) solidity which follows from the exactness and the Akemann--Ostrand property was already known for the free unitary quantum groups, the free orthogonal quantum groups~\citelist{\cite{Ver1}\cite{Vae1}}, and the Kac-type quantum automorphism groups~\cite{Bra2}.

In fact, the non-injectivity assumption is redundant for the case of $\FU_F$, as follows from Theorem \ref{ThmNonInj}, a result due to S. Vaes. Let us record this fact as a separate corollary.

\begin{Cor}
Let $F \in M_N(\C)$ be an invertible matrix. Then the von Neumann algebra $\vN(\FU_{F})$ has no Cartan subalgebra.
\end{Cor}

Note also that for a matrix $F$ with $F \overline{F} \not\in \R I_N$, the quantum group $\FO_F$ is a nontrivial free product, and then~\cite{BHR1}*{Theorem~A} already tells us that there is no Cartan subalgebra in $\vN(\FO_F)$.  When the size of $F$ is $2$, $\vN(\FU^+_F)$ is isomorphic to a free Araki--Woods factor~\cite{DeC1} and does not have a Cartan subalgebra by~\cite{Hou1}.

The von Neumann algebra $L(\FO_F)$ is known to be full, hence noninjective, when $F$ satisfies $\norm{F}^2 \leqslant \Tr(F F^*) / \sqrt{5}$~\cite{Vae1}.  However, in the full
generality the noninjectivity of $L(\FO_F)$ seems to be unsettled (note that the problem of whether the injectivity implies the coamenability of a compact quantum
group is still open).  For now, let us give just another sufficient condition for the noninjectivity of free orthogonal quantum groups.

\begin{Prop}\label{PropNoninjFreeQGrps}
Let $F$ be a matrix of size $N$ satisfying $\Tr(F F^*) = \Tr((F F^*)^{-1})$.  The von Neumann algebra $\vN(\FO_F)$ is not injective if $F$ satisfies $N^2 > \Tr(F
F^*)+2$.
\end{Prop}

\begin{proof}
By Connes's theorem~\cite{Con1}, a von Neumann algebra $M$ on a Hilbert space $\Hsp$ is injective if and only if the map
\[
M \otimes_{\mathrm{alg}} M' \rightarrow B(\Hsp), \quad x \otimes y \mapsto x y
\]
is continuous with respect to the injective norm.

In general, when $\G$ is a compact quantum group, there is a natural isomorphism from $L^\infty(\G)$ to $L^\infty(\G)'$ given by $x \mapsto J \hat{J} x \hat{J} J$.  Here, $J$ is the
modular conjugation of $L^\infty(\G)$ and $\hat{J}$ is the one for $L(\G)$.  These operators can be determined by the formulas
\begin{align*}
J \Lambda_\varphi (x) &= f_{\frac{1}{2}}(x_{(1)}^*) f_{\frac{1}{2}}(x_{(3)}^*) \Lambda_\varphi (x_{(2)}^*),&
\hat{J} \Lambda_\varphi(x) &= f_{-\frac{1}{2}}(x^*_{(1)}) f_{\frac{1}{2}}(x^*_{(3)}) \Lambda_\varphi (S(x^*_{(2)})) \quad (x \in \Pol(\G)).
\end{align*}

When $x \in \Pol(\G)$, we can consider the operator $\Ad(x) = x_{(1)} J \hat{J} x_{(2)} \hat{J} J$ on $L^2(\G, \varphi)$.  Using the above formulas, if $(u^{(\pi)}_{i j})$ is an irreducible
representation of $\G$ presented as in Section~\ref{sec:prelim}, we see that its character $\chi_\pi = \sum_i u_{ii}^{(\pi)}$ satisfies
\[
\Ad(\chi_\pi)\Lambda_\varphi(1) = \sum_{i, j} Q_{\pi, j}^{-1} \Lambda_{\varphi}\bigl(u^{(\pi)}_{i j} u^{(\pi) *}_{i j}\bigr).
\]
If we take the inner product of this vector with $\Lambda_\varphi(1)$, we obtain $\dim(\pi)^2 / \dim_q(\pi)$. Hence if $L^{\infty}(\G)$ is injective, we have the state \begin{equation}\label{Eq:Character-State} \chi_{\pi}\mapsto \frac{\dim(\pi)^2}{\dim_q(\pi)}\end{equation} on the closure of the character algebra inside $C_r(\G)$.

Now, consider the case of $\G = \Op_F$ with the defining representation $\pi=\pi_{1/2}$.  By Banica's results~\cite{Ban5}*{Proposition~1 and Theorem~1}, we know that the associated character $\chi_{1/2}$ has
norm $2$ in $C_r(\G)$. Hence, with $\mu_d$ denoting again the dilated Chebyshev polynomials, $\chi_{d/2} = \mu_d(\chi_{1/2})$ has norm $\mu_d(2)=d+1$. On the other hand, \[\dim(\pi_{d/2}) = \mu_d(N) \qquad \textrm{ and }\qquad\dim_q(\pi_{d/2})=\mu_d(\Tr(F^*F)),\] by the normalization condition on $F$. From the asymptotic behavior \eqref{Eq:Chebyshev-growth} of the Chebyshev polynomials, $\frac{\dim(\pi_{d/2})^2}{\dim_q(\pi_{d/2})}$ has the same asymptotics as $(q/q_N^2)^d$, where $q+q^{-1} = \Tr(F^*F)$ and $q_N+q_N^{-1} = N$ with $0<q_N,q<1$.  We deduce that the map \eqref{Eq:Character-State} cannot be bounded on the reduced character algebra if we have $q_N^2 > q$. Since the latter condition is equivalent to $q + q^{-1} < q_N^2 + q_N^{-2} = N^2 - 2$, we conclude that $\vN^{\infty}(\G)$ cannot be injective if one has $N^2 > \Tr(F^*F)+2$.
\end{proof}

\begin{Rem}
We remark that the above consideration is based on the argument of~\cite{BMT1}*{Theorem~4.5}, which says that the central functional $\omega(u^{(\pi)}_{i j}) = \delta_{i j} \dim(\pi)
/ \dim_q(\pi)$ defines a state on $C_r(\G)$ if $L^\infty(\G)$ is injective.  For $\G = \SU_q(2)$, this is precisely the state $\psi_0$ which appeared in Section~\ref{sec:Aof-approx}.
\end{Rem}

\section{Central states and the Drinfel'd double}\label{SecDrin}

The centrality of the states in the proof of Proposition~\ref{Prop:ct-pos-state-SUq2} has a more conceptual explanation. As explained in~\cite{Voi1}*{Section 4}, the
representations $\omega_t$ of $\Pol(\SU_q(2))$ extend to representations of the \emph{Drinfel'd double} of $\SU_q(2)$. The aim of this section is to clarify the correspondence between central
states on a compact quantum group and representations of its quantum double in the general setting.

Let $\G$ be a compact quantum group.  Recall that the linear dual  $\Pol(\G)^*$ of $\Pol(\G)$ becomes a $*$-algebra by~\eqref{eq:dual-alg-structure}. We consider a subalgebra $c_{c}(\widehat{\G})$ of $\Pol(\G)^{*}$ given by all
linear
functionals of the form
\[
y \rhd \varphi\colon x \mapsto \varphi(x y) \in \Pol(\G).
\]
It forms a (non-unital) associative $*$-subalgebra of the space $\Pol(\G)^{*}$. This $*$-algebra is isomorphic to a direct sum of finite-dimensional matrix algebras. In particular,
$\varphi$ becomes a minimal self-adjoint central projection, with $\varphi\omega = \omega(1)\varphi = \omega\varphi$.

The natural pairing between $\Pol(\G)$ and $c_{c}(\widehat{\G})$ leads to the notion of the \emph{Drinfel'd double}. This is a (non-unital, associative) $*$-algebra structure
$\Pol_{c}(\dDd(\G))$ with the underlying vector space $\Pol(\G) \otimes c_{c}(\widehat{\G})$, such that the natural embeddings of $\Pol(\G)$ and $c_{c}(\widehat{\G})$ become $*$-subalgebras
(of its multiplier algebra),
and such that following \emph{interchange law} is satisfied:
\[
\omega(\ccdot x_{(2)})x_{(1)} = x_{(2)}\omega(x_{(1)}\ccdot)
\]
where we denote elementary tensors as $x\omega$. Note that the above interchange law implies the relation
\[
\omega x = x_{(2)}\omega(x_{(1)}\ccdot S(x_{(3)})).
\]

The above $*$-algebra can in fact be made into a \emph{$*$-algebraic quantum group}~\cite{Dra1} by means of the tensor coproduct. It therefore admits a universal C$^*$-envelope
$C_{0}^{u}(\dDd(\G))$~\cite{Kus1} (one can prove this more directly for this particular case). Symbolically, the locally compact quantum group $\Dd(\G)$ such that
$\Pol_{c}(\dDd(\G))$ is the convolution algebra of $\Dd(\G)$ is called the \emph{Drinfel'd double} of $\G$. The $\Dd(\G)$-modules (that is, the $\Pol_{c}(\dDd(\G))$-modules) are
equivalent to the $\G$-Yetter--Drinfeld modules~\cite{Maj1}.

For every linear functional $\omega$ on $\Pol(\G)$, we define a linear functional $\widetilde{\omega}$ on $\Pol_{c}(\dDd(\G))$ by $\widetilde{\omega}(x \theta) = \theta(1)
\omega(x)$.
Then, $\omega \mapsto \widetilde{\omega}$ defines an embedding of $\Pol(\G)^{*}$ into $\Pol_{c}(\dDd(\G))^{*}$, which we denote by $\Ind$. Note that the image of $\Ind$ can be
characterized as the set of elements $\widetilde{\omega}$ such that $\widetilde{\omega}(a) = \widetilde{\omega}(a\varphi)$ for all $a\in \Pol_{c}(\dDd(\G))$.

\begin{Theorem}\label{TheoPrin}
A unital linear functional $\omega$ on $\Pol(\G)$ is a central state if and only if $\widetilde{\omega} = \Ind(\omega)$ is positive on $\Pol_{c}(\dDd(\G))$.
\end{Theorem}

\begin{proof}
Let $\omega$ be a central state on $\Pol(\G)$. Then, for every $y\in \Pol(\G)$ and $\theta,\gamma \in c_{c}(\widehat{\G})$, we compute
\begin{equation*}
\widetilde{\omega}(\theta^{*} y \gamma) = \widetilde{\omega}(y_{(2)}\theta^{*}(y_{(1)}\ccdot S(y_{(3)}))\gamma) = \omega(y_{(2)})\theta^*(y_{(1)}S(y_{(3)}))\gamma(1).
\end{equation*}
By centrality of $\omega$, the right-hand side is equal to $\overline{\theta(1)}\gamma(1)$. Hence, we get
\begin{eqnarray*}
\widetilde{\omega}\Bigl(\Bigl(\sum_{i} x_{i}\theta_{i}\Bigr)^{*}\Bigl(\sum_{i} x_{i}\theta_{i}\Bigr)\Bigr) & = & \sum_{i,j}
\omega(x_{i}^{*}x_{j})\overline{\theta_{i}(1)}\theta_{j}(1) \\
& = & \omega\Bigl(\Bigl(\sum_{i}\theta_{i}(1)x_{i}\Bigr)^{*}\Bigl(\sum_{i}\theta_{i}(1)x_{i}\Bigr)\Bigr),
\end{eqnarray*}
which is positive by assumption on $\omega$.

Conversely, assume that $\omega$ is a unital linear functional such that $\Ind(\omega) = \widetilde{\omega}$ is positive. As noticed already, $\widetilde{\omega}(a) =
\widetilde{\omega}(a\varphi)$ for all $a \in \Pol_{c}(\dDd(\G))$. As $\Pol_{c}(\dDd(\G))^{2} = \Pol_{c}(\dDd(\G))$ and $\widetilde{\omega}$ is positive, we see that
$\widetilde{\omega}$ is Hermitian. Hence, $\widetilde{\omega}(\varphi a) = \widetilde{\omega}(a)$ for all $a\in \Pol_{c}(\dDd(\G))$. In particular, $\widetilde{\omega}(\varphi x) =
\omega(x)$ for all $x\in \Pol(\G)$. As $\omega(x^{*}x) = \widetilde{\omega}(\varphi x^{*}x\varphi)$ for $x\in \Pol(\G)$, we see that $\omega$ is positive.

Let us show the centrality of $\omega$. Using the Hermitian property, one has
\[
\widetilde{\omega}(\psi x) = \overline{\widetilde{\omega}(x^{*} \psi^{*})} = \psi(1) \omega(x) \quad x \in \Pol(\G).
\]
Combining this with the interchange law, we obtain
\begin{equation*}
(\theta\omega)(x) = \widetilde{\omega}(\theta(\ccdot x_{(2)})x_{(1)}) = \widetilde{\omega}(x_{(2)}\theta(x_{(1)}\ccdot)) = (\omega\theta)(x).
\end{equation*}
As the pairing between $\Pol(\G)$ and $c_{c}(\widehat{\G})$ is non-degenerate, it follows that $\omega$ is central.
\end{proof}

It is not difficult to show that the $*$-representation of $\Pol_{c}(\dDd(\G))$ in its GNS-representation with respect to any positive functional is necessarily bounded. It hence
follows that the problem of finding central states on $\Pol(\G)$ is equivalent to the problem of finding $*$-representation of $C_{0}^{u}(\dDd(\G))$ admitting a $\varphi$-fixed
vector.

\begin{Rem} The C$^*$-algebra $C_0^u(\dDd(\SU_{q}(2)))$ may be interpreted as the universal enveloping C$^*$-algebra of the quantum Lorentz group $\SL_{q}(2, \C)$ \cite{Pod1}. The
representations of Voigt which appear in Proposition~\ref{Prop:ct-pos-state-SUq2} could thus be interpreted as representations of $\SL_{q}(2, \C)$, which are analogues of the
complementary series representations of $\SL(2,\C)$. To the best of our knowledge, the full representation theory of $\SL_{q}(2, \C)$ is not completely understood
yet.
\end{Rem}

\begin{Rem}
There is no analogue of Theorem \ref{TheoPrin} for \emph{non-positive}, bounded central functionals on $\Pol(\G)$ (with respect to the universal C$^*$-norm). Indeed, one can show that
there is a $*$-isomorphic copy of the character algebra of $\G$ inside $\Pol(\dDd(\G))$, given by sending $\chi_{\pi} = \sum_r u_{rr}^{(\pi)}$ to $\sigma_{-i/2}(\chi_{\pi})\varphi$,
where $\sigma_{-i/2}(u_{rs}^{\pi}) = Q_{\pi,r}^{1/2}Q_{\pi,s}^{1/2}u_{rs}^{(\pi)}$. For the case of $\G = \SU_q(2)$, combining the results of this section with the ones from
section \ref{sec:Aof-approx}, one can then show that the resulting C$^*$-algebra closure of the character algebra inside $\Pol(\dDd(\G))$ is isomorphic to $C(-|q+q^{-1}|,|q+q^{-1}|)$.
It follows from this that the functional $\theta_z$ of Theorem \ref{Thm:Psi-z-slice-by-bdd-psi-z} does \emph{not} extend to a bounded functional on $C_0^u(\dDd(\SU_q(2)))$ when $z$ is
not real. This implies that a bounded central functional on $\Pol(\G)$ (with respect to the universal norm) does not always decompose into a linear combination of central
\emph{states}.
\end{Rem}

\appendix
\renewcommand{\thesection}{A.\arabic{section}}
\setcounter{section}{0}

\section*{\center{Appendix}\\ Fullness and factoriality for the free unitary quantum group von Neumann algebras}

\begin{center}by {\large Stefaan Vaes}\end{center}

\bigskip

In the following, $F$ is an arbitrary invertible complex matrix.  Using~\cite{Ban6}, the equivalence classes of irreducible objects of $\Rep(\Up_F)$ can be identified with the free semigroup\linebreak $\N \star \N$ generated by $\alpha$ and $\beta$, respectively corresponding to the generating corepresentation and to its contragredient.  The duality on $\Rep(\Up_F)$ is encoded as the antimultiplicative involution $w \mapsto \bar{w}$ on $\N \star \N$ induced by $\bar{\alpha} = \beta$.  The fusion rules of $\Rep(\Up_F)$ are given by
\[
w \cdot v = \sum_{x y = w, \bar{y} z = v} x z,
\]
where $x y$ denotes the usual concatenation product of $x$ and $y$ in $\N \star \N$.

We consider the GNS Hilbert space $L^2(\Up_F,\varphi)$ with respect to the Haar state $\varphi$ on $L(\FU_F)$ and we view $L(\FU_F)$ as a vector subspace of $L^2(\Up_F,\varphi)$. So we have $\langle a,b\rangle = \varphi(b^* a)$ for all $a,b \in L(\FU_F)$.

For every $x \in \N \star \N$, we fix a corepresentation $U_x \in B(H_x) \otimes L(\FU_F)$ representing $x$. Using the orthogonality relations \eqref{eq:inn-prod}, we find positive invertible matrices $Q_x \in B(H_x)$ such that $\mathrm{Tr}(Q_x) = \dim_q(x) = \mathrm{Tr}(Q_x^{-1})$ and such that the corresponding state $\omega_x$ on $B(H_x)$ given by $\omega_x(T) = \frac{1}{\dim_q(x)}\mathrm{Tr}(TQ^{-1}_x)$ satisfies
$$(\iota \otimes \varphi) (U_x^* (T \otimes 1) U_x) = \omega_x(T) 1 \quad\text{for all}\;\; T \in B(H_x) \; .$$
The modular automorphism group $(\sigma_t^\varphi)_{t \in \R}$ of $\varphi$ is determined by
$$(\id\otimes \sigma_t^{\varphi})(U_x) = (Q_x^{it}\otimes 1)U_x(Q_x^{it}\otimes 1) \; .$$
It then follows that
\begin{equation}\label{eq:invariance-U}
(\omega_x \otimes \varphi)(U_x(1 \otimes a)U_x^*) = \varphi(a) \quad\text{for all}\;\; a \in L(\FU_F) \; .
\end{equation}
The scalar product $\langle T,S \rangle = \omega_x(S^*T)$ turns $B(H_x)$ into a Hilbert space that we denote as $K_x$. For every $t \in \R$, we define
\[
V_x^t\colon L^2(\Up_F, \varphi) \rightarrow K_x \otimes L^2(\Up_F, \varphi) : V_x^t a \mapsto ( \iota \otimes \sigma^\varphi_t)(U_x) (1 \otimes a) U_x^* \; .
\]
Using \eqref{eq:invariance-U}, it follows that $V_x^t$ is an isometry.

The following analogue of Puk\'{a}nszky's $14 \epsilon$-argument is the key to the fullness and factoriality of $L(\FU_F)$ and also provides the computation of the Sd-invariant of $L(\FU_F)$.

\begin{Prop}
\label{prop:14epsilon-arg}
For all $a \in L(\FU_F)$ and $t \in \R$, we have
\[
\norm{a-\varphi(a)1}_{2,\varphi} \le 14 \max \left (\norm{1 \otimes a - V_{\alpha \beta}^t a}, \norm{1 \otimes a - V_{\alpha^2 \beta}^t a}, \norm{1 \otimes a - V_{\beta \alpha}^t a} \right) \; ,
\]
with the norms at the right hand side being the Hilbert space norms on $K_x \otimes L^2(\Up_F,\varphi)$.
\end{Prop}

\begin{proof}
For every $x \in \N \star \N$, we denote by $C_x \subset L(\FU_F)$ the linear span of all the matrix coefficients of $U_x$. For every $v \in \N \star \N$, we then define $H_v$ as the closed linear span of all $C_{vw}$, $w \in \N \star \N$. By construction,
\begin{equation}\label{eq.image}
V^t_x C_y \subset \operatorname{span} \{ K_x \ot C_z \mid z \in \N \star \N \;\;\text{appears in}\;\; x \cdot y \cdot \overline{x}\} \; .
\end{equation}
It follows that
$V^t_{\be \al}(H_\al) \subset K_{\be\al} \ot H_\beta$.

Fix $a \in L(\FU_F)$ and $t \in \R$. Decompose $a$ in the orthogonal decomposition $L^2(\Up_F,\varphi) = \C 1 \oplus H_\al \oplus H_\be$ as
$$a = a_\eps + a_\al + a_\be \quad\text{where}\;\; a_\eps = \vphi(a)1 , a_\al \in H_\al , a_\be \in H_\be \; .$$
Write
$$M = \max \left (\norm{1 \otimes a - V_{\alpha \beta}^t a}, \norm{1 \otimes a - V_{\alpha^2 \beta}^t a}, \norm{1 \otimes a - V_{\beta \alpha}^t a} \right) \; .$$
Observe that
\begin{align}
& |\langle 1 \ot a_\be - V^t_{\be\al} a_\al , V^t_{\be \al} a_\al \rangle| \notag\\
\leq \; & |\langle 1 \ot a - V^t_{\be\al} a , V^t_{\be \al} a_\al \rangle| + |\langle 1 \ot (a_\eps + a_\al) - V^t_{\be\al} (a_\eps + a_\be), V^t_{\be \al} a_\al \rangle|\label{eq.myequ}
\end{align}
Because $V^t_{\be\al}(H_\al) \subset K_{\be\al} \ot H_\be$, we get that $1 \ot (a_\eps + a_\al)$ is orthogonal to $V^t_{\be \al} a_\al$. Since $V^t_{\be\al}$ is an isometry, we also have that $V^t_{\be\al} (a_\eps + a_\be)$ is orthogonal to $V^t_{\be \al} a_\al$. So the last term of \eqref{eq.myequ} is zero and we conclude that
\begin{equation}\label{eq.first-estim}
|\langle 1 \ot a_\be - V^t_{\be\al} a_\al , V^t_{\be \al} a_\al \rangle| \leq M \|a_\al\|_{2,\vphi} \leq M \|a-\vphi(a)1\|_{2,\varphi} \; .
\end{equation}
Writing $1 \ot a_\be$ as the sum of $1 \ot a_\be - V^t_{\be\al} a_\al$ and $V^t_{\be \al} a_\al$, it follows that
\begin{equation}\label{eq.estim-b}
\|a_\be\|_{2,\vphi}^2 \geq \|a_\al\|_{2,\vphi}^2 - 2 M \, \|a-\vphi(a)1\|_{2,\varphi} \; .
\end{equation}

Next, define the isometries $V_1^t, V_2^t\colon L^2(\Up_F, \varphi) \rightarrow K_{\alpha \beta} \otimes K_{\alpha^2 \beta} \otimes L^2(\Up_F, \varphi)$ by setting $V_1^t \xi = (V_{\alpha \beta}^t \xi)_{1 3}$ and $V_2^t \xi = (V_{\alpha^2 \beta}^t \xi)_{2 3}$. It follows from \eqref{eq.image} that $V_1^t(H_\beta)$ and $V_2^t(H_\beta)$ are orthogonal. Using \eqref{eq.first-estim} with the roles of $\al$ and $\be$ interchanged, we then get that
\begin{align*}
|\langle 1 \ot 1 \ot a_\al - V_1^t a_\be - V_2^t a_\be , V_1^t a_\be \rangle| &= |\langle 1 \ot 1 \ot a_\al - V_1^t a_\be , V_1^t a_\be \rangle| \\
&= |\langle 1 \ot a_\al -V_{\al\be}^t a_\be, V_{\al\be}^t a_\be \rangle| \leq M \|a-\vphi(a)1\|_{2,\varphi} \; .
\end{align*}
We similarly have that
$$|\langle 1 \ot 1 \ot a_\al - V_1^t a_\be - V_2^t a_\be , V_2^t a_\be \rangle| \leq M \|a-\vphi(a)1\|_{2,\varphi} \; .$$
Writing $1 \ot 1 \ot a_\al$ as the sum of $V_1^t a_\be$, $V_2^t a_\be$ and $1 \ot 1 \ot a_\al - V_1^t a_\be - V_2^t a_\be$, the previous two estimates, together with the orthogonality of $V_1^t a_\be$ and $V_2^t a_\be$, imply that
\begin{equation}\label{eq.estim-c}
\|a_\al\|_{2,\vphi}^2 \geq 2 \|a_\be\|_{2,\vphi}^2 - 4M \|a-\vphi(a)1\|_{2,\varphi} \; .
\end{equation}
Combining the inequalities in \eqref{eq.estim-b} and \eqref{eq.estim-c}, it follows that
\begin{align*}
& \|a_\al\|_{2,\vphi}^2 \geq 2\|a_\al\|_{2,\vphi}^2 - 8 M \, \|a-\vphi(a)1\|_{2,\varphi} \quad\text{and}\\
& \|a_\be\|_{2,\vphi}^2 \geq 2\|a_\be\|_{2,\vphi}^2 - 6 M \, \|a-\vphi(a)1\|_{2,\varphi} \; .
\end{align*}
Adding up these inequalities and using that $\|a_\al\|_{2,\vphi}^2 + \|a_\be\|_{2,\vphi}^2 = \|a-\vphi(a)1\|^2_{2,\varphi}$, we finally get that
$$\|a-\vphi(a)1\|_{2,\vphi} \leq 14 M \; .$$
\end{proof}

Let $M$ be a von Neumann algebra with separable predual. The group $\Aut M$ of automorphisms of $M$ is a Polish group, with $\al_n \rightarrow \al$ if and only if $\|\om \al_n - \om \al\| \rightarrow 0$ for every $\omega \in M_*$. When the group of inner automorphisms $\Inn M$ is a closed subgroup of $\Aut M$, the von Neumann algebra $M$ is said to be full. A bounded sequence $(x_n)$ in $M$ is called a central sequence if for every $y \in M$, we have that $x_n y - y x_n \rightarrow 0$ in the strong$^*$-topology. The bounded sequence $(x_n)$ is called a strongly central sequence if $\|\om x_n - x_n \om\| \rightarrow 0$ for every $\om \in M_*$. By \cite[Theorem 3.1]{Con2}, $M$ is full if and only if every strongly central sequence $(x_n)$ in $M$ is trivial, meaning that there exists a bounded sequence $z_n$ in the center of $M$ such that $x_n - z_n \rightarrow 0$ in the strong$^*$-topology.

A normal semifinite faithful (n.s.f.) weight $\psi$ on a von Neumann algebra $M$ is called almost periodic if the modular operator $\Delta_\psi$ has pure point spectrum. A factor $M$ is called almost periodic if it admits an almost periodic n.s.f.\ weight. In that case, the intersection of the point spectra of the modular operators $\Delta_\psi$ of all almost periodic n.s.f.\ weights $\psi$ on $M$ is a subgroup of $\R_+^*$ denoted by $\Sd(M)$~; see \cite[Definition 1.2]{Con2}.

\begin{Theorem}\label{ThmNonInj}
Any central sequence in $L(\FU_F)$ is asymptotically scalar. So the von Neumann algebra $L(\FU_F)$ is a full and non-injective factor.

Denoting by $Q$ the unique positive multiple of $FF^*$ satisfying $\Tr(Q) = \Tr(Q^{-1})$, the $\Sd$-invariant of $L(\FU_F)$ is the subgroup of $\R_+^*$ generated by the eigenvalues of $Q \otimes Q$.
\end{Theorem}

\begin{proof}
Write $M = L(\FU_F)$ and let $(x_n)$ be a central sequence in $M$. Applying Proposition~\ref{prop:14epsilon-arg} with $t=0$, we conclude that $\|x_n - \vphi(x_n)1\|_{2,\vphi} \rightarrow 0$. Since also $(x_n^*)$ is a central sequence, also $\|x_n^* - \vphi(x_n^*)1\|_{2,\varphi} \rightarrow 0$. Both together imply that $x_n - \vphi(x_n)1$ converges to $0$ in the strong$^*$-topology. So $(x_n)$ is asymptotically scalar. It follows that $M$ is a full factor. In particular, $M$ is non-injective.

Denote by $\Lambda \subset \R_+^*$ the subgroup generated by the eigenvalues of $Q \otimes Q$. Since the modular group $(\sigma^\vphi_t)_{t \in \R}$ of the Haar state $\vphi$ is given by
$$(\id \ot \sigma_t^\varphi)(U_x) = (Q_x^{it} \ot 1) U_x (Q_x^{it} \ot 1) \; ,$$
it follows that $\vphi$ is almost periodic and that the point spectrum $\Sd(\vphi)$ of the modular operator of $\vphi$ is contained in $\Lambda$. So $\Sd(M) \subset \Sd(\vphi) \subset \Lambda$.

It remains to prove that $\Lambda \subset \Sd(M)$. Since $\Sd(M)$ is a group, it suffices to prove that $\Sd(\vphi) \subset \Sd(M)$. By \cite[Theorem 4.7]{Con2}, we can take an almost periodic n.s.f.\ weight $\psi$ on $M$ such that $\Sd(M) = \Sd(\psi)$. We apply \cite[Proposition 1.1]{Con2} to the subgroup $\Sd(\psi) = \Sd(M)$ of $\R_+^*$. In order to show that $\Sd(\vphi) \subset \Sd(M)$, it then suffices to prove the following statement~: if $t_n$ is a sequence in $\R$ such that $\sigma_{t_n}^\psi \rightarrow \id$ in $\Aut(M)$, then also $\sigma_{t_n}^\varphi \rightarrow \id$ in $\Aut(M)$. Fix such a sequence $t_n$. Put $v_n = [D\vphi:D\psi]_{t_n}$. Then $v_n$ is a sequence of unitaries in $M$ such that $(\Ad v_n^*) \circ \sigma_{t_n}^\varphi = \sigma_{t_n}^\psi \rightarrow \id$.

We now use the notation of Proposition \ref{prop:14epsilon-arg}. Since convergence in $\Aut(M)$ implies pointwise convergence in $\|\,\cdot\,\|_{2,\varphi}$, it follows that for all $x \in \N \star \N$ and with respect to the Hilbert norm on $K_x \otimes L^2(M,\vphi)$, we have
$$\|(\id \ot \sigma^\vphi_{t_n})(U_x) (1 \ot v_n) U_x^* - 1 \ot v_n \| \rightarrow 0 \; .$$
Applying Proposition \ref{prop:14epsilon-arg} to $a = v_n$, we get that $\|v_n - \vphi(v_n)1\|_{2,\vphi} \rightarrow 0$. This means that $|\vphi(v_n)| \rightarrow 1$. So we find a sequence $\lambda_n \in \C$ with $|\lambda_n| = 1$ for all $n$ and $|\vphi(v_n) - \lambda_n| \rightarrow 0$. It follows that $\|v_n - \lambda_n 1\|_{2,\vphi} \rightarrow 0$. Since also $|\vphi(v_n^*) - \overline{\lambda_n} | = |\vphi(v_n) - \lambda_n | \rightarrow 0$, we also get that $\|v_n^* - \overline{\lambda_n} 1\|_{2,\vphi} \rightarrow 0$. So we have proved that $v_n - \lambda_n 1$ converges to $0$ in the strong$^*$-topology. It follows that $\Ad v_n \rightarrow \id$ in $\Aut(M)$. Since also $(\Ad v_n^*) \circ \sigma^\vphi_{t_n} \rightarrow \id$, we conclude that $\sigma^\vphi_{t_n} \rightarrow \id$ in $\Aut(M)$, and the theorem is proved.
\end{proof}

\raggedright

\end{document}